\documentclass[11pt,b5paper,notitlepage]{article}
\usepackage[b5paper, margin={0.5in,0.65in}]{geometry}

\usepackage{amsmath,amscd,amssymb,amsthm,mathrsfs,amsfonts,layout,indentfirst,graphicx,caption,mathabx, stmaryrd,appendix,calc,imakeidx,upgreek,appendix} 
\usepackage[dvipsnames]{xcolor}
\usepackage{palatino}  
\usepackage{slashed} 
\usepackage{mathrsfs} 
\usepackage{extarrows} 
\usepackage{enumitem} 
\usepackage[all]{xy}
\usepackage{fancyhdr} 
\usepackage{verbatim}
\usepackage{booktabs, tabularx, array}
\newcolumntype{Y}{>{\centering\arraybackslash}X}

\usepackage{tikz-cd}
\usepackage[nottoc]{tocbibind}   

\usepackage{float}

\usepackage{lipsum}
\let\OLDthebibliography\thebibliography
\renewcommand\thebibliography[1]{
	\OLDthebibliography{#1}
	\setlength{\parskip}{0pt}
	\setlength{\itemsep}{2pt} 
}

\allowdisplaybreaks  
\usepackage{latexsym}
\usepackage{chngcntr}
\usepackage[colorlinks,linkcolor=blue,anchorcolor=blue, linktocpage,
]{hyperref}
\hypersetup{ urlcolor=cyan,
	citecolor=[rgb]{0,0.5,0}}

\usepackage{simpler-wick}

\DeclareFontFamily{U}{rcjhbltx}{}
\DeclareFontShape{U}{rcjhbltx}{m}{n}{<->s*[1.15]rcjhbltx}{}   
\DeclareSymbolFont{hebrewletters}{U}{rcjhbltx}{m}{n}

\DeclareMathSymbol{\lamed}{\mathord}{hebrewletters}{108}
\DeclareMathSymbol{\mem}{\mathord}{hebrewletters}{109}
\DeclareMathSymbol{\ayin}{\mathord}{hebrewletters}{96}
\DeclareMathSymbol{\tsadi}{\mathord}{hebrewletters}{118}
\DeclareMathSymbol{\qof}{\mathord}{hebrewletters}{113}
\DeclareMathSymbol{\shin}{\mathord}{hebrewletters}{152}

\counterwithin{figure}{section}

\theoremstyle{definition}
\newtheorem{df}{Definition}[section]

\newtheorem{rem}[df]{Remark}

\theoremstyle{plain}
\newtheorem{thm}[df]{Theorem}

\newtheorem{pp}[df]{Proposition}
\newtheorem{co}[df]{Corollary}

\newcommand{\fk}{\mathfrak}

\newcommand{\wtd}{\widetilde}

\newcommand{\ovl}{\overline}

\newcommand{\Dbb}{\mathbb{D}}

\newcommand{\End}{\mathrm{End}} 
\newcommand{\idt}{\mathbf{1}}
\newcommand{\Hom}{\mathrm{Hom}}

\newcommand{\Res}{\mathrm{Res}}

\newcommand{\opp}{\mathrm{op}}

\newcommand{\SV}{\mathscr{V}}

\newcommand{\Vect}{\mathcal Vect}

\newcommand{\blt}{\bullet}

\newcommand{\Vbb}{\mathbb V}
\newcommand{\Ubb}{\mathbb U}
\newcommand{\Xbb}{\mathbb X}
\newcommand{\Ybb}{\mathbb Y}
\newcommand{\Wbb}{\mathbb W}
\newcommand{\Mbb}{\mathbb M}
\newcommand{\Gbb}{\mathbb G}
\newcommand{\Cbb}{\mathbb C}
\newcommand{\Nbb}{\mathbb N}
\newcommand{\Zbb}{\mathbb Z}
\newcommand{\Pbb}{\mathbb P}

\newcommand{\Ebb}{\mathbb E}
\newcommand{\cbf}{\mathbf c}
\newcommand{\wt}{\mathrm{wt}}

\newcommand{\<}{\left\langle}
\renewcommand{\>}{\right\rangle}
\newcommand{\MO}{\mathcal{O}}
\newcommand{\MU}{\mathcal{U}}

\newcommand{\fx}{\mathfrak{X}}

\newcommand{\ST}{\mathscr{T}}

\newcommand{\bk}[1]{\langle {#1}\rangle}

\newcommand{\Mod}{\mathrm{Mod}}

\newcommand{\id}{\mathrm{id}}

\newcommand{\eps}{\varepsilon}

\newcommand{\fn}{\mathfrak{N}}
\newcommand{\fy}{\mathfrak{Y}}

\newcommand{\Irr}{\mathrm{Irr}}

\newcommand{\Lan}{{\big\langle}}
\newcommand{\Ran}{{\big\rangle}}

\newcommand{\Abb}{{\mathbb{A}}}

\newcommand{\Del}{\mathrm{Del}}

\usepackage{tipa} 

\usepackage{tipx}

\numberwithin{equation}{section}

\title{Non-Equivalence of Smooth and Nodal Conformal Block Functors in Logarithmic CFT}
\author{{\sc Hao Zhang}
}
\date{}
\begin{document}\sloppy 
	\pagenumbering{arabic}
	\setcounter{section}{-1}

	\maketitle

\newcommand\blfootnote[1]{%
	\begingroup
	\renewcommand\thefootnote{}\footnote{#1}%
	\addtocounter{footnote}{-1}%
	\endgroup
}




\begin{abstract}
Let $\Vbb$ be an $\Nbb$-graded, $C_2$-cofinite vertex operator algebra (VOA) admitting a non-lowest generated module in $\Mod(\Vbb)$ (e.g., the triplet algebras $\mathcal{W}_p$ for $p\in \Zbb_{\geq 2}$ or the even symplectic fermion VOAs $SF_d^+$ for $d\in \Zbb_+$). We prove that, unlike in the rational case, for $N\geq 4$ there is no vector bundle on $\overline{\mathcal{M}}_{0,N}$ whose fibers are the spaces of conformal blocks associated to certain $\Vbb$-modules, since the dimensions of these conformal block spaces differ between smooth and nodal curves. Consequently, the sheaf of coinvariants associated to these $\Vbb$-modules on $\overline{\mathcal{M}}_{0,N}$ is not locally free for $N\geq 4$. It also follows that, unlike in the rational case, the mode transition algebra $\fk A$ introduced by Damiolini-Gibney-Krashen \cite{DGK2,DGK-presentations} is not isomorphic to the end $\Ebb=\int_{\Xbb\in \Mod(\Vbb)}\Xbb\otimes \Xbb'$ as an object of $\Mod(\Vbb^{\otimes 2})$.
\end{abstract}

\tableofcontents





	
	

	

\section{Introduction}
Let $\Vbb$ be an $\Nbb$-graded, $C_2$-cofinite vertex operator algebra (VOA). When $\Vbb$ is rational, a remarkable achievement is the factorization property of conformal blocks associated to grading-restricted $\Vbb$-modules \cite{DGT2}. In fact, a \textit{sewing theorem} in the algebraic geometric setting is proved in \cite[Thm. 8.5.1]{DGT2}. This theorem establishes isomorphisms, called \textbf{(formal) sewing isomorphisms}, relating spaces of conformal blocks of higher genus or fewer marked points (e.g., $\fx$ in \eqref{intro1} or $\fy$ in \eqref{intro4}) to those of lower genus or more marked points (e.g., $\wtd \fx$ in \eqref{intro1} or $\wtd \fy$ in \eqref{intro4}):
\begin{subequations}
\begin{gather}
	\label{intro1}\fx=\vcenter{\hbox{{
					   \includegraphics[height=1.5cm]{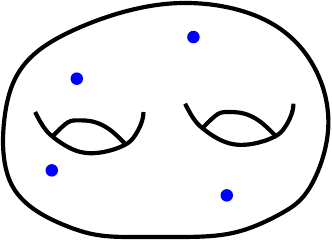}}}},\quad 
						\wtd\fx=\vcenter{\hbox{{
										   \includegraphics[height=1.5cm]{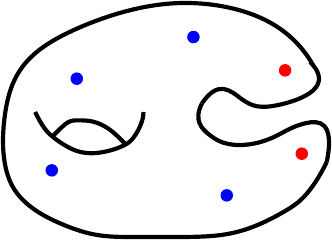}}}}\\
	\label{intro4}\fy=\vcenter{\hbox{{
		\includegraphics[height=1.5cm]{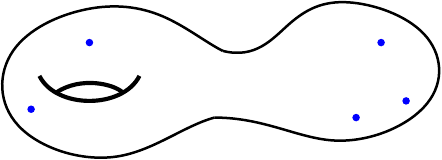}}}},\quad 
		 \wtd\fy=\vcenter{\hbox{{
							\includegraphics[height=1.5cm]{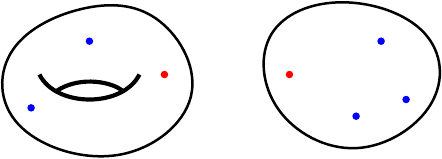}}}}
	\end{gather}
\end{subequations}
 These sewing isomorphisms are typically constructed through conformal blocks for \textit{nodal curves} (e.g., $\fx_0$ in \eqref{intro2} or $\fy_0$ in \eqref{intro5}) \cite{TUY,BFM-conformal-blocks,Ueno97,NT-P1_conformal_blocks,Ueno08,DGT1,DGT2}.
\begin{subequations}
\begin{gather}
	\label{intro2}\fx_0=\vcenter{\hbox{{
											\includegraphics[height=1.5cm]{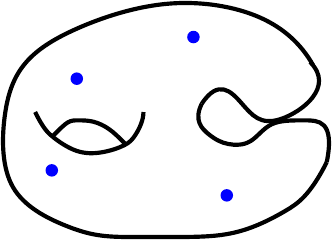}}}}\\
	\label{intro5}\fy_0=\vcenter{\hbox{{
		\includegraphics[height=1.5cm]{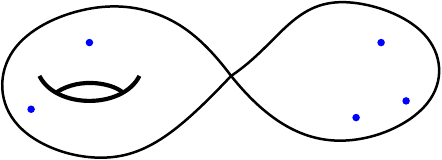}}}}
\end{gather}
\end{subequations}
The sewing theorem implies that the dimensions of the spaces of conformal blocks for nodal curves (e.g., $\fx_0$ or $\fy_0$) and for smooth curves (e.g. $\fx$ or $\fy$) coincide, a condition that ensures a vector bundle on $\ovl{\mathcal{M}}_{g,N}$ whose fibers are spaces of conformal blocks (see \cite[Sec. 8]{DGT2} for details). In this paper, we show that, once the rationality assumption is removed, this dimension constancy—and hence the vector bundle structure—no longer holds for general $\Nbb$-graded, $C_2$-cofinite VOAs.

From now on, let $\Vbb$ be an $\Nbb$-graded, $C_2$-cofinite VOA admitting a module in $\Mod(\Vbb)$ that \textit{is not generated by its lowest weight subspace}. (It can be proved that the triplet algebras $\mathcal{W}_p$ for $p\in \Zbb_{\geq 2}$ \cite{Kausch1991,AM-triplet,NagatomoTsuchiya2011,TW-triplet} and the even symplectic fermion VOAs $SF_d^+$ for $d\in \Zbb_+$ \cite{Kausch1995,GaberdielKausch1999,Kausch2000,Runkel2014,GR-symplectic,FGR-symplectic} satisfy this condition; see Cor. \ref{triplet} and \ref{symplectic}.) Recently, Damiolini-Gibney-Krashen introduced the notion of \textbf{strong identity} property, a sufficient condition that guarantees the vector bundle structure on $\ovl{\mathcal{M}}_{g,N}$ \cite[Cor. 5.2.6]{DGK2}. Moreover, it was proved that the triplet algebras $\mathcal{W}_p$ do not satisfy the strong identity property \cite[Prop. 9.1.4]{DGK2}. This provides evidence, though it does not directly imply, that the dimensions of spaces of their conformal blocks differ between nodal and smooth curves.

Motivated by \cite{DGK2}, we obtain our main result (Thm. \ref{main1}), stating that there exist $\Xbb,\Ybb\in \Mod(\Vbb)$ such that 
\begin{align}\label{intro3}
	\dim \ST^*\Bigg(\vcenter{\hbox{{
					   \includegraphics[height=2.0cm]{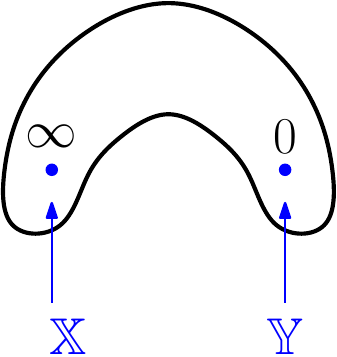}}}}~\Bigg)\ne 
						\dim \ST^*\Bigg(\vcenter{\hbox{{
										   \includegraphics[height=2.0cm]{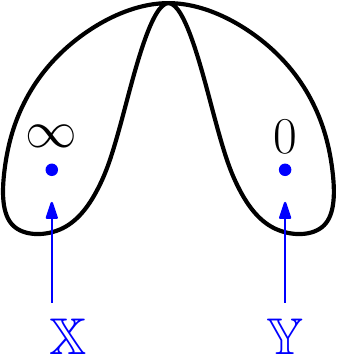}}}}~\Bigg)
	\end{align}
where $\ST^*(\cdots)$ denotes the space of conformal blocks. In fact, our proof shows that $\Ybb$ can be chosen to be an indecomposable projective object, and $\Xbb$ can be chosen to be either an indecomposable projective object or a simple object. Our main idea is a variation of the following argument: the end $\Ebb=\int_{\Xbb\in \Mod(\Vbb)}\Xbb\otimes \Xbb'$ and the mode transition algebra $\fk A$ \cite{DGK3-morita,DGK2,DGK-presentations} are objects of $\Mod(\Vbb^{\otimes 2})$ representing the respective conformal block functors on the two sides of \eqref{intro3}. However, $\fk A$ is generated by its lowest weight subspace, while $\Ebb$ is not; see Thm. \ref{mainlowest}.

There are several further consequences of \eqref{intro3}, summarized as follows. 
\begin{itemize}
	\item In \cite[Question 5.6.3]{DW-modular-functor}, the authors ask about the relationship between the mode transition algebra $\fk A$ and the end $\Ebb$. When $\Vbb$ is $C_2$-cofinite and rational, it is proved in \cite{DGK2} that $\fk A\simeq \Ebb$ as objects in $\Mod(\Vbb^{\otimes 2})$. However, in our case, it follows immediately from \eqref{intro3} that
	\begin{align*}
		\fk A\not \simeq \Ebb=\int_{\Xbb\in \Mod(\Vbb)}\Xbb\otimes \Xbb'\quad  \text{in }\Mod(\Vbb^{\otimes 2}).
	\end{align*}
	See Thm. \ref{main8} for details.
	\item In \cite[Question 6.1]{DGK-presentations}, the authors ask whether the sheaf of coinvariants is a vector bundle. By \eqref{intro3} and propagation of conformal blocks, there exist $\Xbb,\Ybb\in \Mod(\Vbb)$ such that
	\begin{align*}
		\dim \ST^*\Bigg(\vcenter{\hbox{{
						   \includegraphics[height=2.5cm]{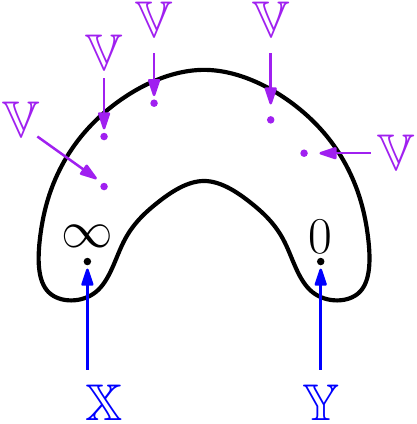}}}}~\Bigg)\ne 
							\dim \ST^*\Bigg(\vcenter{\hbox{{
											   \includegraphics[height=2.5cm]{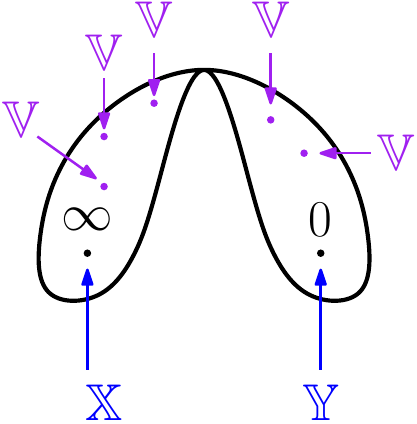}}}}~\Bigg)
		\end{align*}
	Therefore,
	\begin{enumerate}
		\item[(a)] there is no vector bundle on $\ovl{\mathcal{M}}_{0,N}$ whose fibers are spaces of conformal blocks associated to $\Xbb,\Ybb,\underbrace{\Vbb,\cdots,\Vbb}_{N-2\text{ copies}}$ for $N\geq 4$.
		\item[(b)] the sheaf of coinvariants associated to $\Xbb,\Ybb,\underbrace{\Vbb,\cdots,\Vbb}_{N-2\text{ copies}}$ is not locally free on $\ovl{\mathcal{M}}_{0,N}$ for $N\geq 4$. 
	\end{enumerate}
	See Rem. \ref{locallyfree} for details.
\end{itemize}

We remark that our result does not rule out the possibility that the two dimensions appearing in \eqref{intro3} become equal if the definition of nodal conformal blocks is suitably modified. However, if such a new definition were to exist, one would not expect it to have a direct connection with the Zhu algebra $A(\Vbb)$ \cite{Zhu-modular-invariance}, since in the non-semisimple setting $A(\Vbb)$ exerts only weak control over modules that are not generated by their lowest weight subspaces.

\subsection*{Acknowledgments}
I am grateful to Chiara Damiolini for introducing the connection between nodal conformal blocks and mode transition algebras, to Xu Gao for answering many questions on mode transition algebras, and to Bin Gui for many enlightening conversations and valuable comments on this project. I would also like to thank the referee for the careful and detailed report, which helped me improve the manuscript.

\section{Preliminaries}
\subsection{Notation}
\label{notation}
Throughout this paper, we use the following notation.
\begin{itemize}
	\item $\Nbb=\{0,1,2,\cdots\}$, $\Zbb_+=\{1,2,\cdots\}$, $\Zbb_{\geq 2}=\{2,3,\cdots\}$.
	\item Let $X$ be a finite set. Then $\mathrm{Card}(X)$ denotes the cardinality of $X$.
	\item Let $\Re(z)$ be the real part of $z\in \Cbb$.
	\item Let $\Vect$ be the category of finite dimensional vector spaces over $\Cbb$.
	\item Let $\zeta$ be the standard coordinate of $\Cbb$.
	\item Throughout this paper, we fix an $\Nbb$-graded $C_2$-cofinite vertex operator algebra (VOA) $\Vbb=\bigoplus_{n\in \Nbb} \Vbb(n)$ with conformal vector $\cbf$ and vacuum vector $\idt$. Each nonzero vector $v$ in $\Vbb(n)$ is \textbf{homogeneous} of weight $\wt(v)=n$.
	\item For each $N\in \Nbb$, let $\Mod(\Vbb^{\otimes N})$ denote the category of grading-restricted generalized $\Vbb^{\otimes N}$-modules. $\Mod(\Vbb^{\otimes N})$ is an abelian category by \cite{Hua-projectivecover} (see also \cite{MNT10}). The comparison between $\Vbb^{\otimes N}$-modules and $\Vbb^{\times N}$-modules can be found in \cite{GZ1}.
	\item Since $\Vbb$ is $C_2$-cofinite, it has only finitely many equivalence classes of irreducible $\Vbb$-modules (see \cite[Prop. 4.2]{Hua-projectivecover}). Denote by $\Irr$ a finite set of representatives of these classes.
	\item Let $\Xbb,\Ybb \in \Mod(\Vbb)$, and assume that $\Xbb$ is irreducible. Then $[\Ybb:\Xbb]$ denotes the multiplicity of $\Xbb$ in a composition series of $\Ybb$.
	\item If $\Wbb\in\Mod(\Vbb^{\otimes N})$ and $v\in\Vbb$, then the $i$-th vertex operator
	\begin{align*}
	Y_{\Wbb,i}(v,z)=\sum_{n\in\Zbb}Y_{\Wbb,i}(v)_nz^{-n-1}
	\end{align*}
	is $Y(\idt\otimes\cdots\otimes v\otimes\cdots\otimes\idt,z)$ with $v$ placed in the $i$-th component. We abbreviate $Y_{\Wbb,i}$ to $Y_i$ when no confusion arises. We also define
		\begin{align}\label{prime}
		Y_i'(v,z)=Y_i(\MU(\upgamma_z)v,z^{-1})
		\end{align}
		where $\MU(\upgamma_z)=e^{zL(1)}(-z^{-2})^{L(0)}$. Moreover, we expand
	\begin{align*}
	Y_i'(v,z)=\sum_{n\in\Zbb}Y_i'(v)_nz^{-n-1}
	\end{align*}
	Finally, we set $L_i(n)=Y_i(\cbf)_{n+1}$. 
    \item If $\Wbb\in\Mod(\Vbb^{\otimes N})$ and $\lambda_1,\dots,\lambda_N\in\Cbb$, then $\Wbb_{[\lambda_\blt]}$ is the subspace of all $w\in\Wbb$ such that for all $1\leq i\leq N$, $w$ is a generalized eigenvector of $L_i(0)$ with eigenvalue $\lambda_i$. Any $w\in \Wbb_{[\lambda_\blt]}$ is said to be \textbf{homogeneous} of weight $\lambda_\blt$. In this case, we write 
	\begin{gather*}
		\wt_i(w)=\lambda_i,\quad 1\leq i\leq N,\\
		\wt(w)=\sum_{i=1}^N \wt_i(w)=\sum_{i=1}^N \lambda_i.
	\end{gather*}
	The finite dimensional subspace $\Wbb_{[\leq\lambda_\blt]}$ is defined to be the direct sum of all $\Wbb_{[\mu_\blt]}$ where $\Re(\mu_i)\leq \Re(\lambda_i)$ for all $1\leq i\leq N$. Then the contragredient $\Vbb^{\otimes N}$-module of $\Wbb$, as a vector space, is
	\begin{align*}
	\Wbb'=\bigoplus_{\lambda_\blt\in\Cbb^N}(\Wbb_{[\lambda_\blt]})^*
	\end{align*}
	Then for each $w\in\Wbb,w'\in\Wbb'$ we clearly have
	\begin{align*}
	\bk{Y_i(v,z)w,w'}=\bk{w,Y_i'(v,z)w'}
	\end{align*}
	Finally, the algebraic completion of $\Wbb$ is 
	\begin{align*}
	\ovl\Wbb=(\Wbb')^*=\prod_{\lambda_\blt\in\Cbb^N}\Wbb_{[\lambda_\blt]}
	\end{align*}
	\item If $\Wbb\in\Mod(\Vbb^{\otimes 2})$, we refer to the vertex operators $Y_1$ and $Y_2$ as $Y_+$ and $Y_-$, respectively, i.e.,
	\begin{align*}
		Y_+(v,z)=Y(v\otimes \idt,z),\quad Y_-(v,z)=Y(\idt\otimes v,z)
	\end{align*}
	for $v\in \Vbb$. We shall also denote $\Wbb$ by $(\Wbb,Y_+,Y_-)$. If $w\in \Wbb$ is homogeneous of weight $(\lambda_+,\lambda_-)$, then $\wt_+(w)=\lambda_+,\wt_-(w)=\lambda_-$ and $\wt(w)=\lambda_++\lambda_-$. 
\end{itemize}

\subsection{Non-lowest generated $\Vbb$-modules in $\Mod(\Vbb)$}
In this section, we focus on modules in $\Mod(\Vbb)$ that are not generated by their lowest weight subspaces. 

\begin{df}
Let $\Xbb$ be a weak $\Vbb$-module. Define
\begin{align*}
	\Omega(\Xbb)=\{w\in \Xbb:Y(v)_n w=0 \text{ for all homogeneous }v\in \Vbb,n\in \Zbb\text{ such that }\wt (v)-n-1<0 \}
\end{align*}
We refer to $\Omega(\Xbb)$ as the \textbf{lowest weight subspace} of $\Xbb$. The module $\Xbb$ is said to be \textbf{lowest generated} if it is generated by $\Omega(\Xbb)$.
\end{df}

\begin{rem}\label{lb5}
	Let $\Xbb$ be a weak $\Vbb$-module, and let $\mathcal{A}$ denote the subalgebra of $\End(\Xbb)$ generated by all operators $Y(v)_n$ where $v\in \Vbb,n\in \Zbb$. Let $\mathcal{A}_{\geq 0}$ be the unital subalgebra generated by all operators $Y(v)_n$ such that $\wt(v)-n-1\geq 0$. If $\mathcal{T}$ is a subspace of $\Omega(\Xbb)$, then it is easy to see that $\mathcal{A}\cdot \mathcal{T}=\mathcal{A}_{\geq 0}\cdot \mathcal{T}$. 
\end{rem}

\begin{df}\label{lb7}
	Let $\Xbb \in \Mod(\Vbb)$.
\begin{enumerate}
	\item[(a)] We say that $\Xbb$ is \textbf{singly generated} (resp. \textbf{singly lowest generated}) if there exists a homogeneous vector $w\in \Xbb$ (resp. \textbf{$w\in \Omega(\Xbb)$}) such that $\Xbb$ is generated by $w$. 
	\item[(b)] Assume that $\Xbb$ is singly generated. The \textbf{conformal weight} $\wt(\Xbb)$ of $\Xbb$ is defined to be $\wt(x)$ where $x\in \Xbb$ is homogeneous and satisfies $\Re(\wt(y))\geq \Re(\wt(x))$ for all homogeneous $y\in \Xbb$.
\end{enumerate}
\end{df}

\begin{rem}\label{lb6}
	Let $\Xbb \in \Mod(\Vbb)$.
	\begin{enumerate}
		\item[(a)] Assume that $\Xbb$ is singly lowest generated. Let $w\in \Omega(\Xbb)$ be any homogeneous vector generating $\Xbb$. Then Rem. \ref{lb5} implies that $\wt(\Xbb)=\wt(w)$. 
		\item[(b)] Assume that $\Xbb$ is singly lowest generated. Let $\Ybb$ be a nonzero quotient $\Vbb$-module of $\Xbb$. Then $\Ybb$ is also singly lowest generated, and $\wt(\Ybb)=\wt(\Xbb)$. This is because the homogeneous vector generating $\Xbb$ must be sent by the quotient map to a homogeneous vector generating $\Ybb$.
	\end{enumerate}
\end{rem}

The following two propositions are standard elementary facts. See also \cite[Prop. 5.4]{DLMtwisted}.
\begin{pp}
	If $\Xbb$ is an irreducible $\Vbb$-module, then it is singly lowest generated with conformal weight $\wt(\Xbb)=:\alpha$. Moreover, $\Omega(\Xbb)$ coincides with the subspace $\Xbb_{[\alpha]}$ consisting of eigenvectors of $L(0)$ with eigenvalue $\alpha$. 
\end{pp}
\begin{proof}
	Since any nonzero homogeneous vector $w\in \Omega(\Xbb)$ generates $\Xbb$, it follows from Rem. \ref{lb6}-(a) that $\alpha=\wt(w)$. Thus $\Omega(\Xbb)\subset \Xbb_{[\alpha]}$. The reverse inclusion $\Omega(\Xbb)\supset \Xbb_{[\alpha]}$ is obvious.
\end{proof}

\begin{pp}\label{lb1}
	Let $\Xbb\in \Mod(\Vbb)$. Then $\Omega(\Xbb)=\oplus_{\lambda\in \Cbb}\Omega(\Xbb)_{[\lambda]}$. Consequently, $\Omega(\Xbb)$ is spanned by a set of homogeneous vectors. 
\end{pp}
\begin{proof}
	Using the decomposition $\Xbb=\oplus_{\lambda\in \Cbb}\Xbb_{[\lambda]}$, each $w\in \Omega(\Xbb)$ can be written as $w=\sum_{\lambda\in \Cbb}w_{\lambda}$, where $w_\lambda\in \Xbb_{[\lambda]}$. For each $n\in \Zbb$ and homogeneous $v\in \Vbb$ such that $\wt(v)-n-1<0$, noting that $Y(v)_n w=0$ and that $Y(v)_n w_{\lambda}$ is the projection of $Y(v)_n w $ onto the $\Xbb_{[\wt(v)-n-1+\lambda]}$, we have $w_\lambda\in \Omega(\Xbb)_{[\lambda]}$. This proves the decomposition $\Omega(\Xbb)=\oplus_{\lambda\in \Cbb}\Omega(\Xbb)_{[\lambda]}$.
\end{proof}

We need projective covers of irreducible $\Vbb$-modules to study non-lowest generated $\Vbb$-modules. By \cite[Thm. 3.23]{Hua-projectivecover}, together with
\cite[Prop. 4.1 and 4.3]{Hua-projectivecover}, we have the following theorem.
\begin{thm}[{\cite{Hua-projectivecover}}]\label{projective}
	For each irreducible $\Vbb$-module $\Xbb\in \Mod(\Vbb)$, there exists a \textbf{projective cover} $\varphi_\Xbb:P_\Xbb\rightarrow \Xbb$ (abbreviated as $P_\Xbb$ when no confusion arises), i.e.,
	\begin{enumerate}
		\item[(a)] $P_\Xbb$ is a projective $\Vbb$-module in $\Mod(\Vbb)$.
		\item[(b)] $\varphi_\Xbb$ is an epimorphism in $\Mod(\Vbb)$. 
		\item[(c)] If $f:\Mbb\rightarrow P_\Xbb$ is a morphism in $\Mod(\Vbb)$ such that the composition
		\begin{align*}
			\Mbb \xrightarrow{f} P_\Xbb\xrightarrow{\varphi_\Xbb} \Xbb
		\end{align*}
		is an epimorphism, then $f$ itself is an epimorphism.
	\end{enumerate}
    Moreover, $P_\Xbb$ is unique up to isomorphism in $\Mod(\Vbb)$. 
\end{thm}

By \cite[Thm. 3.24]{Hua-projectivecover}, together with
\cite[Prop. 4.1]{Hua-projectivecover}, we have the following theorem.
\begin{thm}[{\cite{Hua-projectivecover}}]\label{finitelinear}
    For each $N\geq 1$, the category $\Mod(\Vbb^{\otimes N})$ is a finite $\Cbb$-linear category.
\end{thm}

\begin{rem}\label{composition}
	Let $\Xbb, \Ybb \in \Mod(\Vbb)$, and assume that $\Xbb$ is irreducible. Recall the definition of $[\Ybb : \Xbb]$ from Sec. \ref{notation}. It is well known from the theory of abelian categories that $\dim\Hom_\Vbb(P_\Xbb,\Ybb)=[\Ybb:\Xbb]$ (cf. \cite[(1.7)]{EGNO}).
\end{rem}

\begin{pp}\label{lb2}
	Let $\Xbb$ be an irreducible $\Vbb$-module. Then its projective cover $P_\Xbb$ is singly generated. Therefore, by Def. \ref{lb7}, the conformal weight $\wt(P_\Xbb)$ is well-defined.
\end{pp}
\begin{proof}
	Let $w\in P_\Xbb$ be a nonzero homogeneous vector such that $\varphi_\Xbb(w)\ne 0$. Define $\Wbb$ to be the submodule of $P_\Xbb$ generated by $w$. Then the composition
	\begin{align}\label{eq4}
		\Wbb\hookrightarrow P_\Xbb \twoheadrightarrow \Xbb
	\end{align}
	is nonzero. Since $\Xbb$ is irreducible, \eqref{eq4} must be surjective. By Thm. \ref{projective}-(c), it follows that the inclusion $\Wbb \hookrightarrow P_\Xbb$ is surjective. Hence $P_\Xbb = \Wbb$.
\end{proof}
\begin{rem}\label{lb4}
	Let $\Xbb$ be an irreducible $\Vbb$-module. Then $\wt(\Xbb)\in \wt(P_\Xbb)+\Nbb$. In particular, 
	\begin{align*}
		\Re(\wt(P_\Xbb))\leq \Re(\wt(\Xbb)).
	\end{align*}
\end{rem}
\begin{proof}
    Let $\varphi_\Xbb:P_\Xbb\twoheadrightarrow \Xbb$ be the projective cover of $\Xbb$. By Prop. \ref{lb2}, $P_\Xbb$ is generated by one
homogeneous vector. Hence all generalized $L(0)$-weights of $P_\Xbb$ lie in a
single coset modulo $\Zbb$: indeed, if $v\in \Vbb$ is homogeneous, then the
mode $Y(v)_n$ shifts weights by $\wt(v)-n-1\in \mathbb Z$. Since $\wt(P_\Xbb)$ is the weight of $P_\Xbb$ with minimal real part,
it follows that every weight of $P_\Xbb$ lies in $\wt(P_\Xbb)+\Nbb$.

It remains to show that $\wt(\Xbb)$ is a weight of $P_\Xbb$.
Note that the morphism $\varphi_\Xbb$ preserves generalized $L(0)$-weight spaces. Since
$\Xbb_{[\wt(\Xbb)]}\neq 0$ and $\varphi_\Xbb$ is surjective, we have $(P_\Xbb)_{[\wt(\Xbb)]}\neq 0$. Thus $\wt(\Xbb)$ is a weight of $P_\Xbb$. Therefore, $\wt(\Xbb)\in \wt(P_\Xbb)+\Nbb$.
\end{proof}

The following criterion should be known to experts, although we cannot locate a precise reference. 

\begin{pp}\label{lb3}
	Let $\Xbb\in \Mod(\Vbb)$ be an irreducible $\Vbb$-module. The following conditions are equivalent:
	\begin{enumerate}
		\item[(a)] $P_\Xbb$ is not lowest generated.
		\item[(b)] $\Re(\wt(P_\Xbb))<\Re(\wt(\Xbb))$.
		\item[(c)] $P_\Xbb$ has a composition factor $\Ybb\in \Mod(\Vbb)$ with $\Re(\wt(\Ybb))<\Re(\wt(\Xbb))$.
	\end{enumerate}
\end{pp}
\begin{proof}
	Assume (b). We show that $P_\Xbb$ is not lowest generated. Suppose instead that $P_\Xbb$ is lowest generated. By Prop. \ref{lb1}, there exists a set of homogeneous vectors $\{w_i\in \Omega(P_\Xbb):i\in I\}$ spanning $\Omega(P_\Xbb)$, and for each $i \in I$ let $\Wbb_i$ be the submodule of $P_\Xbb$ generated by $w_i$. Then $\Wbb_i$ is singly lowest generated for each $i\in I$. Since $P_\Xbb$ is lowest generated, there exists an epimorphism 
	\begin{align*}
	\sum_{i\in I}\Wbb_i\twoheadrightarrow P_\Xbb
    \end{align*}
	Hence, for some $i$, the composition $\Wbb_i \hookrightarrow P_\Xbb \twoheadrightarrow \Xbb$ is nonzero, and therefore surjective because $\Xbb$ is irreducible. By Thm. \ref{projective}-(c), the inclusion $\Wbb_i \hookrightarrow P_\Xbb$ is an epimorphism. It follows from Rem. \ref{lb6}-(b) that $P_\Xbb$ is singly lowest generated. Noting that $\Xbb$ is a nonzero quotient $\Vbb$-module of $P_\Xbb$, it follows from Rem. \ref{lb6}-(b) again that $\wt(\Xbb)=\wt(P_\Xbb)$, contradicting (b). Therefore, (a) holds.

	Assume (a). We show that $\Re(\wt(P_\Xbb)) \ne \Re(\wt(\Xbb))$. Suppose instead that $\Re(\wt(P_\Xbb)) = \Re(\wt(\Xbb))$. By Rem. \ref{lb4}, this forces $\wt(P_\Xbb) = \wt(\Xbb)=:\alpha$ and hence $(P_\Xbb)_{[\alpha]}\subset \Omega(P_\Xbb)$. Note that $\varphi_\Xbb$ restricts to a surjective map
	\begin{align*}
		\varphi_\Xbb:(P_\Xbb)_{[\alpha]}\twoheadrightarrow \Xbb_{[\alpha]}
	\end{align*}
	Since $\Xbb_{[\alpha]}$ is nonzero, there exists $0\ne w\in (P_\Xbb)_{[\alpha]}$ with $\varphi_\Xbb(w)\ne 0$. Let $\Wbb$ be the submodule of $P_\Xbb$ generated by $w$. Then the composition $\Wbb\hookrightarrow P_\Xbb\twoheadrightarrow \Xbb$ is nonzero, hence surjective because $\Xbb$ is irreducible. By Thm. \ref{projective}-(c), the inclusion $\Wbb\hookrightarrow P_\Xbb$ is surjective, so $P_\Xbb$ is generated by $w\in (P_\Xbb)_{[\alpha]}\subset \Omega(P_\Xbb)$. Thus $P_\Xbb$ is lowest generated, contradicting (a). Therefore, (b) holds.

	In general, for any $\Mbb\in \Mod(\Vbb)$, the set of weights of $\Mbb$ coincides with the set of weights of its composition factors. Hence
	\begin{align*}
		\Re(\wt(P_\Xbb))=\min_{\Ybb} \Re(\wt(\Ybb))
	\end{align*} 
	where $\Ybb$ ranges over all composition factors of $P_\Xbb$. Therefore, (b) and (c) are equivalent.
\end{proof}

The next proposition follows from a standard projective-cover argument. We
include the proof because this is the precise form needed later.
\begin{pp}\label{lowest2}
	The following are equivalent:
	\begin{enumerate}
		\item[(a)] There exists $\Mbb\in \Mod(\Vbb)$ that is not lowest generated.
		\item[(b)] There exists an irreducible $\Vbb$-module $\Xbb\in \Mod(\Vbb)$ such that $P_\Xbb$ is not lowest generated.
	\end{enumerate}
\end{pp}
\begin{proof}
	It suffices to show that (a) implies (b). Assume (a). Suppose instead that $P_\Xbb$ is lowest generated for every irreducible $\Xbb \in \Mod(\Vbb)$. Let 
	\begin{align*}
		\Gbb:=\oplus_{\Xbb\in \Irr} P_\Xbb
	\end{align*}
	where $\Irr$ is a finite set of representatives of equivalence classes of irreducibles in $\Mod(\Vbb)$. Then $\Gbb$ is lowest generated and serves as a projective generator of $\Mod(\Vbb)$. Hence for any $\Mbb \in \Mod(\Vbb)$ there exists $n \in \Zbb_+$ with an epimorphism $\Gbb^{\oplus n} \twoheadrightarrow \Mbb$, implying that $\Mbb$ is lowest generated. Thus every object in $\Mod(\Vbb)$ would be lowest generated, contradicting (a).
\end{proof}

By applying Prop. \ref{lb3} to the triplet $W$-algebras and the even symplectic fermion VOAs, we obtain the following two corollaries.

\begin{co}\label{triplet}
	Let $p \ge 2$ be an integer, and let $\mathcal{W}_p$ denote the triplet $W$-algebra. Then there exists a module $\Mbb \in \Mod(\mathcal{W}_p)$ that is not lowest generated.
\end{co}
\begin{proof}
	Since $\mathcal{W}_p$ is $C_2$-cofinite \cite{AM-triplet}, all of the preceding results for a general $C_2$-cofinite VOA $\Vbb$ apply. The abelian category structure of $\Mod(\mathcal{W}_p)$ is due to \cite{NagatomoTsuchiya2011}. We follow the terminology of \cite[Sec. 3.1]{TW-triplet}. Up to isomorphism, the irreducible $\mathcal{W}_p$-modules are $X_s^+$ and $X_s^-$ for $1 \le s \le p$. Among them, $X_1^-$ is the unique module with the maximal conformal weight 
	\begin{align*}
		\wt(X_1^-)=\frac{1}{4p}((2p-1)^2-(p-1)^2).	
	\end{align*}
	 Let $P_1^-$ denote its projective cover. The socle series of $P_1^-$ is given by
  \begin{align*}
	X_1^-=S_0(P_1^-)\subset S_1(P_1^-)\subset S_2(P_1^-)=P_1^-
  \end{align*}
  with $S_1(P_1^-)/S_0(P_1^-)\simeq X_{p-1}^+\oplus X_{p-1}^+$ and $S_2(P_1^-)/S_1(P_1^-)\simeq X_1^-$. Since
  \begin{align*}
	\frac{1}{4p}(1-(p-1)^2)=\Re(\wt(X_{p-1}^+))<\Re(\wt(X_1^-))=\frac{1}{4p}((2p-1)^2-(p-1)^2),
  \end{align*}
  condition (c) of Prop. \ref{lb3} holds, and hence $P_1^-$ is not lowest generated.
\end{proof}

\begin{co}\label{symplectic}
	For $d\in \Zbb_+$, let $SF_{d}^+$ denote the even symplectic fermion VOA. Then there exists $\Mbb\in \Mod(SF_{d}^+)$ that is not lowest generated.
\end{co}
\begin{proof}
	Since $SF_{d}^+$ is $C_2$-cofinite \cite{Abe-triplet}, all of the preceding results for a general $C_2$-cofinite VOA $\Vbb$ apply. We follow the terminology of \cite[Sec. 5]{McR-deligne}. Noting that $SF_{1}^+\simeq \mathcal{W}_2$ \cite{GaberdielKausch1999,Kausch2000}, it suffices to consider $d\geq 2$ in our proof. Let $X_1^\pm,X_2^\pm$ denote the irreducible $\mathcal{W}_2$-modules described in the proof of Cor. \ref{triplet}. Recall that $SF_d^+$ is an extension of $\mathcal{W}_2^{\otimes d}$. As a $\mathcal{W}_2^{\otimes d}$-module, we have the decomposition
	\begin{align}\label{symplectic3}
		SF_d^+=\bigoplus_{\mathrm{Card}(\{1\leq i\leq d:\varepsilon_i=-\})\text{ even}} X_1^{\varepsilon_1}\otimes \cdots \otimes X_1^{\varepsilon_d}
	\end{align}
	Up to isomorphism, the irreducible $SF_{d}^+$-modules are $\mathcal{X}_i^\varepsilon$ for $i=1,2$ and $\varepsilon=\pm$. By \cite[Thm. 5.4]{McR-deligne}, the projective cover $\mathcal{P}_1^-$ of $\mathcal{X}_1^-$ has a composition factor isomorphic to the vacuum module $\mathcal{X}_1^+$. Moreover, as a $\mathcal{W}_2^{\otimes d}$-module, we have the equivalence
	\begin{align}\label{symplectic2}
		\mathcal{X}_1^-\simeq SF_d^+\boxtimes (X_1^-\otimes X_1^+\otimes \cdots \otimes X_1^+)
	\end{align}
	Using the decomposition \eqref{symplectic3} and the fusion product $X_1^{\varepsilon_1}\boxtimes X_1^{\varepsilon_2}=X_1^{\varepsilon_1\varepsilon_2}$ in $\Mod(\mathcal{W}_2)$, together with the fact that $\wt(X_1^+)=0$ and $\wt(X_1^-)=1$, one shows that $\wt(\mathcal{X}_1^-)= 1$. Hence, 
	\begin{align*}
		\wt(\mathcal{X}_1^-)=1>0=\wt(\mathcal{X}_1^+)
	\end{align*}
	By Prop. \ref{lb3}, $\mathcal{P}_1^-$ is not lowest generated. This completes the proof.
\end{proof}

\subsection{A dimension criterion}
\begin{pp}\label{dimension}
	Let $\varphi_\Xbb:P_\Xbb\twoheadrightarrow \Xbb$ be the projective cover of an irreducible $\Vbb$-module $\Xbb$ in $\Mod(\Vbb)$. Let $\Ybb\in \Mod(\Vbb)$. If $\dim \Hom_\Vbb(P_\Xbb,\Mbb)=\dim \Hom_\Vbb(\Ybb,\Mbb)$ for all $\Mbb\in \Mod(\Vbb)$, then $P_\Xbb\simeq \Ybb$ in $\Mod(\Vbb)$.
\end{pp}
\begin{proof}
	 Since 
	 \begin{align*}
		\dim \Hom_\Vbb(\Ybb,\Xbb)=\dim \Hom_\Vbb(P_\Xbb,\Xbb)=1
	 \end{align*}
	 there exists a nonzero morphism $\alpha:\Ybb\rightarrow \Xbb$. As $\Xbb$ is irreducible, $\alpha$ must be an epimorphism. By the projectivity of $P_\Xbb$, there exists $\beta:P_\Xbb\rightarrow \Ybb$ such that $\varphi_\Xbb=\alpha\circ \beta$. 

	 We first show that $\beta$ is surjective. Suppose, to the contrary, that $\beta$ is not surjective. Then the nonzero quotient $\Vbb$-module $\Ybb/\beta(P_\Xbb)$ admits an epimorphism onto some irreducible $\Vbb$-module $\Ubb$. 
	 \begin{itemize}
		\item If $\Ubb\not\simeq \Xbb$, then $\dim \Hom_\Vbb(P_\Xbb,\Ubb)=0$. However, since $\Ubb$ is a quotient of $\Ybb/\beta(P_\Xbb)$, we have $\dim \Hom_\Vbb(\Ybb,\Ubb)>0$, contradicting our assumption.
		\item If $\Ubb\simeq \Xbb$, then the composition
		\begin{gather}\label{dimension1}
			\Ybb \twoheadrightarrow \Ybb/\beta(P_\Xbb) \twoheadrightarrow \Ubb\simeq \Xbb
		\end{gather}
		yields a nonzero morphism $\gamma:\Ybb\rightarrow \Xbb$. We claim that $\alpha,\gamma$ are linearly independent in $\Hom_\Vbb(\Ybb,\Xbb)$. To see this, suppose instead that they are linearly dependent. Then, up to a nonzero scalar multiplication, $\varphi_\Xbb=\alpha\circ\beta$ coincides with
		\begin{gather*}
			P_\Xbb\xrightarrow{\beta} \Ybb\twoheadrightarrow \Ybb/\beta(P_\Xbb) \twoheadrightarrow \Ubb\simeq \Xbb
		\end{gather*}
		which is zero, a contradiction. Hence $\alpha,\gamma$ are linearly independent, and so 
		\begin{align*}
		\dim \Hom_\Vbb(\Ybb,\Xbb)\geq 2>1=\dim \Hom_\Vbb(P_\Xbb,\Xbb)
	    \end{align*}
		which is again a contradiction.
	 \end{itemize}
	 Therefore, $\beta$ must be surjective.

	 We now show that $\beta$ is injective. Suppose otherwise. Then the transpose of 
	 \begin{align*}
		P_\Xbb\xrightarrow{\beta} \Ybb \xrightarrow{\alpha} \Xbb
	 \end{align*}
	 is the sequence
	 \begin{align}\label{eq7}
		\Xbb'\xrightarrow{\alpha^t} \Ybb'\xrightarrow{\beta^t}P_\Xbb'
	 \end{align}
	 where $\alpha^t$ is injective and $\beta^t$ is injective but not surjective. Thus we may regard \eqref{eq7} as 
	 \begin{align*}
		\Xbb'\subset \Ybb'\subsetneq P_\Xbb'
	 \end{align*}
	 It follows that the composition series of $P_\Xbb'$ is strictly longer than that of $\Ybb'$. Hence there exists an irreducible $\Vbb$-module $\Wbb$ such that $[\Ybb':\Wbb]<[P_\Xbb':\Wbb]$, equivalently, 
	 \begin{align*}
	 \dim \Hom_\Vbb(P_\Wbb,\Ybb')<\dim \Hom_\Vbb(P_\Wbb,P_\Xbb')
	\end{align*}
	 (See Rem. \ref{composition}). This implies 
	 \begin{align*}
	 \dim \Hom_\Vbb(\Ybb,P_\Wbb')\ne \dim \Hom_\Vbb(P_\Xbb,P_\Wbb')
	\end{align*}
	contradicting our assumption. Therefore, $\beta$ must be injective, and we conclude that $\beta$ is an isomorphism, i.e., $P_\Xbb\simeq \Ybb$.
\end{proof}

\subsection{Smooth conformal block functors}
\label{smooth3}
Let $N\in \Zbb_+$, and let $\fx$ be an $N$-pointed smooth connected rational curve (with local coordinates), i.e.,
\begin{align*}
	\fx=(\Pbb^1\big|x_1,\cdots,x_N;\eta_1,\cdots,\eta_N)
\end{align*}
where $x_1,\cdots,x_N$ are distinct marked points of $\Pbb^1$ and each $\eta_i$ is a local coordinate at $x_i$. Let $\Wbb_1,\cdots,\Wbb_N\in \Mod(\Vbb)$ and associate $\Wbb_i$ to $x_i$ for each $i$. 

The \textbf{space of smooth conformal blocks associated to $\fx$ and $\Wbb_1,\cdots,\Wbb_N$}, denoted
\begin{align*}
	\ST_\fx^*(\Wbb_1\otimes \cdots \otimes \Wbb_N)
\end{align*}
consists of linear functionals 
\begin{align*}
	\varphi:\Wbb_1\otimes \cdots\otimes \Wbb_N\rightarrow \Cbb
\end{align*}
that are invariant under the action determined by $\Vbb$ and the smooth curve
$\fx$. More precisely, this action is defined using the chiral Lie algebra. See \cite{FB04,NT-P1_conformal_blocks} and especially
\cite[Sec. 4.3]{DGT1}. It can also be described through the residue action of the
sheaf of VOAs \cite[(3.8)]{Gui-sewingconvergence}.

 In the following discussion, we use the standard pictorial notation for
spaces of smooth conformal blocks. The picture representing $\ST_\fx^*(\Wbb_1\otimes \cdots \otimes \Wbb_N)$ is 
\begin{align*}
	\ST^*\Bigg(\vcenter{\hbox{{
					   \includegraphics[height=2.5cm]{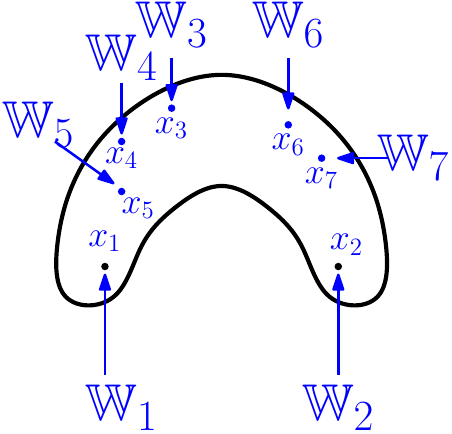}}}}~\Bigg)
	\end{align*}
illustrated here for $N=7$.

\begin{df}\label{smoothcb}
	The contravariant functor
\begin{gather*}
	\ST_{\fx}^*:\Mod(\Vbb)\times \cdots\times \Mod(\Vbb)\rightarrow \Vect\\
	(\Wbb_1,\cdots,\Wbb_N)\mapsto \ST_{\fx}^*(\Wbb_1\otimes \cdots \otimes \Wbb_N)
\end{gather*}
is called the \textbf{smooth conformal block functor associated to $\fx$}. In the pictorial language, $\ST_{\fx}^*$ is also represented as the smooth conformal block functor corresponding to
\begin{align*}
	\vcenter{\hbox{{
				   \includegraphics[height=2.0cm]{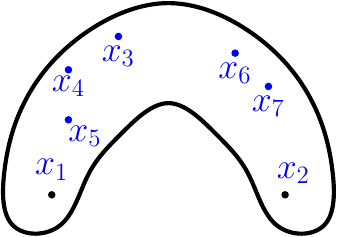}}}}
\end{align*}
illustrated here for $N=7$.
\end{df}

Recall that $\zeta$ is the standard coordinate of $\Cbb$ (see Sec. \ref{notation}). Let $\Xbb,\Ybb\in \Mod(\Vbb)$ and 
\begin{align}\label{smooth1}
	\fn=(\Pbb^1|\infty,0;1/\zeta,\zeta),
\end{align}
where $\Pbb^1$ is identified with $\Cbb\cup\{\infty\}$ and $1/\zeta,\zeta$ are local coordinates at $\infty,0$. The space of smooth conformal blocks associated to $\fn$ and $\Xbb,\Ybb$ 
\begin{align*}
	\ST_{\fn}^*(\Xbb\otimes \Ybb)\equiv \ST^*\Bigg(\vcenter{\hbox{{
					   \includegraphics[height=2.5cm]{fig1a.pdf}}}}~\Bigg)
	\end{align*}
	can be explicitly described by the space of linear functionals $\uppsi:\Xbb\otimes \Ybb\rightarrow \Cbb$ such that: for each $v\in \Vbb,x\in \Xbb, y\in \Ybb$, the relation 
	\begin{align*}
		\uppsi(Y(v,z)x\otimes y)=\uppsi(x\otimes Y'(v,z)y)
	\end{align*}
	holds in $\Cbb[[z^{\pm 1}]]$. It is well known (cf. e.g. \cite[Prop. 5.9.1]{NT-P1_conformal_blocks} or \cite[Prop. 2.3]{GZ5}) that $\ST_{\fn}^*(\Xbb\otimes \Ybb)$ can be identified with $\Hom_\Vbb(\Ybb,\Xbb')$ via the isomorphism 
\begin{gather}\label{eq8}
	\Hom_\Vbb(\Ybb,\Xbb')\xrightarrow{\simeq} \ST_{\fn}^*(\Xbb\otimes \Ybb),\quad T\mapsto T^\flat
\end{gather}
where $T^\flat(x\otimes y)=\Lan x, T(y)\Ran$.

\subsection{Nodal conformal block functors}
\label{nodal2}
Let $N\in \Zbb_{\geq 2}$, and let $\fk Y$ be an $N$-pointed nodal rational curve (with local coordinates), i.e., 
\begin{align*}
	\fk Y=(\mathcal{P}\big|x_1,\cdots,x_N;\eta_1,\cdots,\eta_N)
\end{align*}
where $\mathcal{P}$ is a nodal rational curve (i.e., a nodal curve of genus 0) with one node, $x_1,\cdots,x_N$ are marked points of $\mathcal{P}$ distinct from the node and each $\eta_i$ is a local coordinate at $x_i$. Moreover, we assume that the two irreducible components of $\mathcal{P}$ contain $x_1$ and $x_2$, respectively. Let $\Wbb\in \Mod(\Vbb^{\otimes 2})$ and $\Wbb_3,\cdots,\Wbb_N\in \Mod(\Vbb)$. Associate $\Wbb=(\Wbb,Y_+,Y_-)$ to $x_1,x_2$ via the default ordering (cf. \cite{GZ5})
\begin{align*}
	\eps:\{+,-\}\rightarrow \{x_1,x_2\},\qquad \eps(+)= x_1,\quad\eps(-)=x_2
\end{align*}
and $\Wbb_i$ to $x_i$ for each $3\leq i\leq N$. 

The \textbf{space of nodal conformal blocks associated to $\fk Y$ and $\Wbb,\Wbb_3,\cdots,\Wbb_N$ via the default ordering $\eps$}, denoted
\begin{align*}
	\ST_{\fk Y}^*(\Wbb\otimes \Wbb_3\otimes  \cdots\otimes \Wbb_N)
\end{align*}
consists of linear functionals 
\begin{align*}
	\Wbb\otimes \Wbb_3\otimes  \cdots\otimes \Wbb_N\rightarrow \Cbb
\end{align*}
that are invariant under the action determined by $\Vbb$ and the nodal curve $\fk Y$. More precisely, this action is defined using the chiral Lie algebra. See \cite{NT-P1_conformal_blocks,DGT2} and especially
\cite[Sec. 4.3]{DGT1}. It can also be described through the residue action of the
sheaf of VOAs \cite[(3.8)]{Gui-sewingconvergence}.

In the following discussion, we use the standard pictorial notation for
spaces of nodal conformal blocks. The picture representing $\ST_{\fk Y}^*(\Wbb\otimes \Wbb_3\otimes  \cdots\otimes \Wbb_N)$ is 
\begin{align*}
	\ST^*\Bigg(\vcenter{\hbox{{
					   \includegraphics[height=2.5cm]{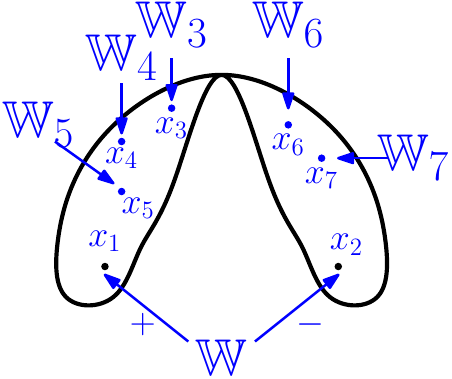}}}}~\Bigg)
	\end{align*}
illustrated here for $N=7$.

In the special case where $\Wbb=\Wbb_1\otimes \Wbb_2$ with $\Wbb_1,\Wbb_2\in \Mod(\Vbb)$, the picture representing $\ST_{\fk Y}^*(\Wbb\otimes \Wbb_3\otimes  \cdots\otimes \Wbb_N)\equiv \ST_{\fk Y}^*(\Wbb_1 \otimes  \cdots\otimes \Wbb_N)$ is given by
\begin{align*}
	\ST^*\Bigg(\vcenter{\hbox{{
					   \includegraphics[height=2.5cm]{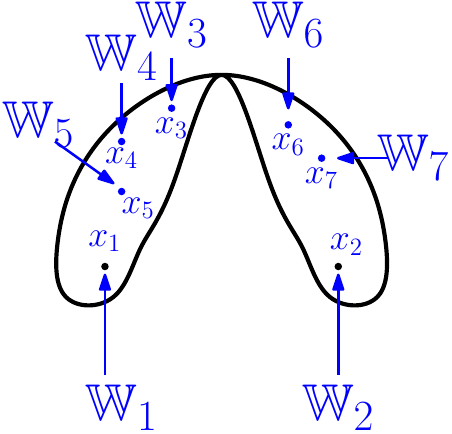}}}}~\Bigg)
	\end{align*}
illustrated here for $N=7$.

\begin{df}\label{nodalcb}
	The contravariant functor 
	\begin{gather*}
		\ST_{\fk Y}^*:\Mod(\Vbb)\times \cdots\times \Mod(\Vbb)\rightarrow \Vect\\
		(\Wbb_1,\cdots,\Wbb_N)\mapsto \ST_{\fk Y}^*(\Wbb_1\otimes \cdots \otimes \Wbb_N)
	\end{gather*}
	is called the \textbf{nodal conformal block functor associated to $\fk Y$}. In the pictorial language, $\ST_{\fk Y}^*$ is also represented as the nodal conformal block functor corresponding to
	\begin{align*}
		\vcenter{\hbox{{
					   \includegraphics[height=2.0cm]{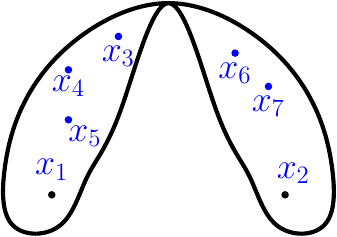}}}}
	\end{align*}
	\end{df}

Consider the 2-pointed nodal rational curve (with local coordinates) 
\begin{align}\label{nodal1}
	\fk B=(\mathcal{B}\big| \infty,0;\xi,\varpi).
\end{align} 
Here $\mathcal{B}$ is the nodal rational curve obtained by gluing two copies of 
\begin{align*}
	\fn=\eqref{smooth1}=(\Pbb^1|\infty,0;1/\zeta,\zeta)
\end{align*}
along the point $0$ on the first copy $\fn^+$ and the point $\infty$ on the second $\fn^-$. The gluing map is denoted by $\nu:\fn^+\sqcup \fn^-\twoheadrightarrow \fk B$ and the node is denoted by $x$. The nodal rational curve $\mathcal{B}$ carries two marked points, $\infty$ and $0$, both distinct from the node $x$: the point $\infty$ is inherited from $\fn^+$, while $0$ is inherited from $\fn^-$. The local coordinates $\xi$ and $\varpi$ are inherited from $1/\zeta$ of $\fn^+$ and $\zeta$ of $\fn^-$. Thus, each irreducible component of $\fk B$ contains exactly one marked point. This construction is illustrated in Fig. \ref{fig:B-gluing-process}.

\begin{figure}[H]
	\centering
	\includegraphics[width=0.68\textwidth]{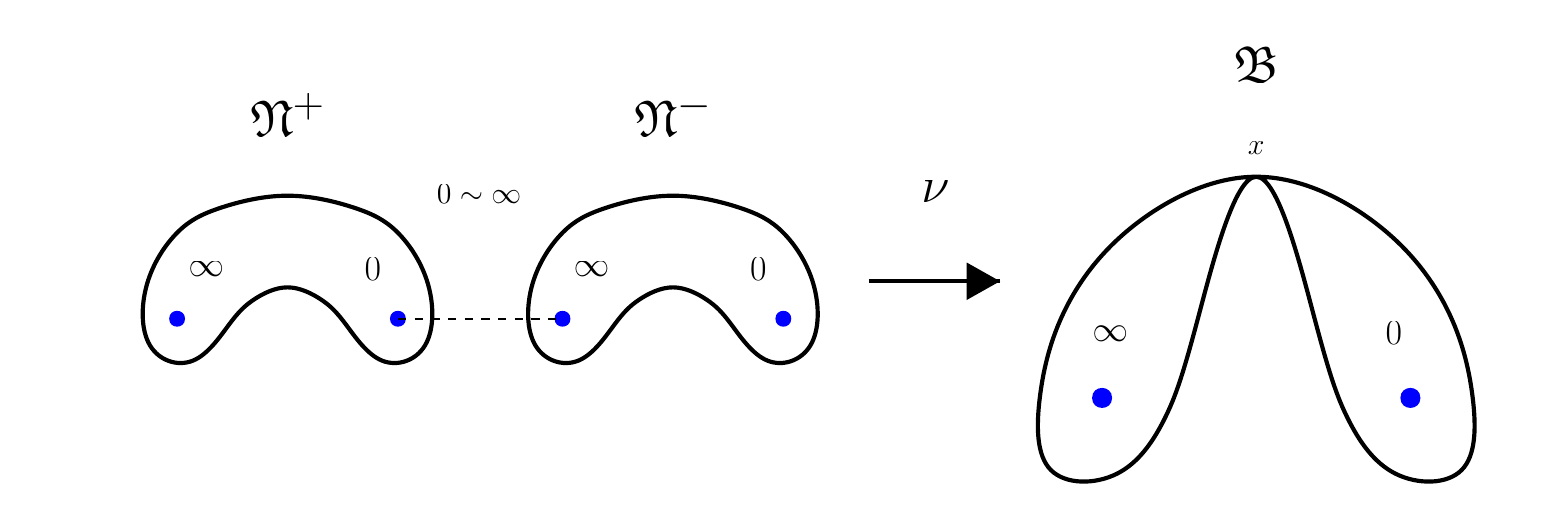}
	\caption{\hspace{0.5em}The construction of $\fk B$ by gluing two copies of $\fn$.}
	\label{fig:B-gluing-process}
\end{figure}

\begin{pp}\label{prop19}
Let $\Wbb\in \Mod(\Vbb^{\otimes 2})$. The space of nodal conformal blocks associated to $\fk B$ and $\Wbb$ via $\eps$
\begin{align*}
	\ST_{\fk B}^*(\Wbb)\equiv\ST^*\Bigg(\vcenter{\hbox{{
					   \includegraphics[height=2.5cm]{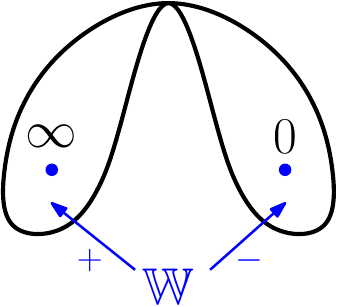}}}}~\Bigg)
	\end{align*}
	can be explicitly described by the space of linear functionals $\uppsi:\Wbb\rightarrow \Cbb$ such that: for each $n\in \Zbb,w\in \Wbb$ and homogeneous $v\in \Vbb$, we have
\begin{subequations}\label{eq17}
\begin{gather}
    \label{eq18}\uppsi(Y_+(v)_{\wt(v)-1} w)=\uppsi(Y_-'(v)_{\wt(v)-1} w)\\
	\label{eq16}\uppsi(Y_+(v)_n w)= \uppsi( Y_-(v)_n  w)=0,\quad \text{ if }\wt(v)-n-1>0.
\end{gather}
\end{subequations}
\end{pp}
\begin{proof}
    We follow \cite[Ch. 2]{Gui-sewingconvergence} to describe the sheaf of VOA 
    \begin{align*}
        \SV_{\mathcal{B}}=\varinjlim_{n\in\Nbb} \SV_{\mathcal{B}}^{\leq n}
    \end{align*}
     on the nodal curve $\mathcal{B}$ in an explicit way. See also \cite{NT-P1_conformal_blocks,DGT1,DGT2}. We shall define $\SV_{\mathcal{B}}^{\leq n}$ to be an $\MO_{\mathcal{B}}$-submodule of $\SV_{\mathcal{B}-\{x\}}^{\leq n}$ (which is the sheaf of VOA on the smooth curve $\mathcal{B}-\{x\}$, cf. \cite{Gui-sewingconvergence}). Outside the node $x$, $\SV_{\mathcal{B}}^{\leq n}$ is equal to $\SV_{\mathcal{B}-\{x\}}^{\leq n}$. To describe $\SV_{\mathcal{B}}^{\leq n}$ near the node $x$, we choose a neighborhood $U$ of $x$ such that $\nu^{-1}(U)=V^+\sqcup V^-$, where $0\in V^+$ is a neighborhood of $0$ in $\fn^+$, and $\infty\in V^-$ is a neighborhood of $\infty$ in $\fn^-$. Then, $U-\{x\}$ can be identified with $(V^+-\{0\})\cup (V^--\{\infty\})$ via $\nu$. Moreover, $\zeta\in \MO(V^+)$ and $1/\zeta\in \MO(V^-)$ are local coordinates and they give the trivializations
     \begin{gather*}
         \MU_\varrho^+(\zeta):\SV_{\mathcal{B}-\{x\}}^{\leq n}|_{V^+-\{0\}}\xrightarrow{\simeq} \Vbb^{\leq n}\otimes_\Cbb \MO_{V^+-\{0\}}\\
         \MU_\varrho^-(1/\zeta):\SV_{\mathcal{B}-\{x\}}^{\leq n}|_{V^--\{\infty\}}\xrightarrow{\simeq} \Vbb^{\leq n}\otimes_\Cbb \MO_{V^--\{\infty\}}
     \end{gather*}
     By \cite[(2.8)]{Gui-sewingconvergence}, $\SV_{\mathcal{B}}^{\leq n}(U)$ is the $\MO_{\mathcal{B}}(U)$-submodule of $\SV_{\mathcal{B}-\{x\}}^{\leq n}(U-\{x\})$ generated by 
     \begin{align*}
         \MU_\varrho^+(\zeta)^{-1}\big(\zeta^{L(0)}v\big)+\MU_\varrho^-(1/\zeta)^{-1}\big((1/\zeta)^{L(0)}\MU(\gamma_1)v\big),\quad v\in \Vbb^{\leq n}
     \end{align*}

     The dualizing sheaf $\omega_{\mathcal{B}}$ is described as follows. Outside the node $x$, $\omega_{\mathcal{B}}$ coincides with $\omega_{\mathcal{B}-\{x\}}$. On the neighborhood $U$ of the node $x$, $\omega_{\mathcal{B}}(U)$ is the free $\MO_{\mathcal{B}}(U)$-submodule of $\omega_{\mathcal{B}-\{x\}}((V^+-\{0\})\cup (V^--\{\infty\}))$ generated by vectors whose restrictions to $V^+-\{0\}$ and $V^--\{\infty\}$ are
     \begin{align*}
         \zeta^{-1}d\zeta\quad \text{and}\quad -\zeta d(1/\zeta).
     \end{align*}

     Let $\SV_{\mathcal{B}}\otimes \omega_{\mathcal{B}}(\blt \infty+\blt 0)$ be the sheaf of $\SV_{\mathcal{B}}\otimes \omega_{\mathcal{B}}$ with finite poles at $\infty$ and $0$. In order to describe conformal blocks, we need to describe global sections of $\SV_{\mathcal{B}}\otimes \omega_{\mathcal{B}}(\blt \infty+\blt 0)$. Define $\sigma_{v},\sigma_{v,n}^+,\sigma_{v,n}^-$ for $n\geq 1$ and $v\in \Vbb$ as follow:
     \begin{align*} \sigma_v &= \MU_\varrho^+(\zeta)^{-1} \big(\zeta^{L(0)-1}v\,d\zeta\big) - \MU_\varrho^-(1/\zeta)^{-1} \big((1/\zeta)^{L(0)-1}\MU(\gamma_1)v\,d(1/\zeta)\big), \\[0.5em] \sigma_{v,n}^+ &= \begin{cases} \MU_\varrho^+(\zeta)^{-1} \big(\zeta^{L(0)+n-1}v\,d\zeta\big), & \text{on } V^+-\{0\},\\ 0, & \text{on } V^--\{\infty\}, \end{cases} \\[0.5em] \sigma_{v,n}^- &= \begin{cases} 0, & \text{on } V^+-\{0\},\\ -\MU_\varrho^-(1/\zeta)^{-1} \big((1/\zeta)^{L(0)+n-1}\MU(\gamma_1)v\,d(1/\zeta)\big), & \text{on } V^--\{\infty\}. \end{cases} \end{align*}
     It is clear that $\sigma_{v},\sigma_{v,n}^+,\sigma_{v,n}^-$ can be extended to global sections of $\SV_{\mathcal{B}}\otimes \omega_{\mathcal{B}}(\blt \infty+\blt 0)$. It remains to prove that they span $H^0(\mathcal{B},\SV_{\mathcal{B}}\otimes \omega_{\mathcal{B}}(\blt \infty+\blt 0))$.
     
    Let $\sigma$ be any such global section. The local description above
and the finite-pole condition give a finite expansion
\begin{align*}
\sigma={}&\sum_{n\geq0}\MU_\varrho^+(\zeta)^{-1}
 \big(\zeta^{L(0)+n-1}v_n^+\,d\zeta\big)\\
&-\sum_{n\geq0}\MU_\varrho^-(1/\zeta)^{-1}
 \big((1/\zeta)^{L(0)+n-1}\MU(\gamma_1)v_n^-\,d(1/\zeta)\big),
\end{align*}
where $v_n^\pm\in\Vbb$ and $v_0^+=v_0^-$, with no restriction on $v_n^\pm$ for $n\geq1$. Thus,
\begin{align*}
\sigma=\sigma_{v_0^+}
+\sum_{n\geq1}\bigl(\sigma_{v_n^+,n}^+
                    +\sigma_{v_n^-,n}^-\bigr).
\end{align*}
Hence, $H^0(\mathcal{B},\SV_{\mathcal{B}}\otimes \omega_{\mathcal{B}}(\blt \infty+\blt 0))$ is spanned by $\sigma_{v},\sigma_{v,n}^+,\sigma_{v,n}^-$ for $n\geq 1$ and $v\in \Vbb$.

     Without loss of generality, assume that $v$ is homogeneous. For each $w\in \Wbb$, by the definition of residue action (see for example \cite[(1.39)]{GZ2})
     \begin{align*}
         \sigma_v\cdot w&=\Res_{\infty}Y_+(\MU_\varrho^+(1/\zeta)\sigma_v,1/\zeta)w+\Res_{0}Y_-(\MU_\varrho^-(\zeta)\sigma_v,\zeta)w\\
         &=-Y_+'(v)_{\wt(v)-1}w+Y_-(v)_{\wt(v)-1}w\\
         \sigma_{v,n}^+\cdot w&=\Res_{\infty}Y_+(\MU_\varrho^+(1/\zeta)\sigma_{v,n}^+,1/\zeta)w=-Y_+'(v)_{\wt(v)+n-1}w\\
         \sigma_{v,n}^-\cdot w&=\Res_{0}Y_-(\MU_\varrho^-(\zeta)\sigma_{v,n}^-,\zeta)w=Y_-(v)_{\wt(v)-n-1}w
     \end{align*}
      Therefore, $\uppsi:\Wbb\rightarrow \Cbb$ is a conformal block in $\ST_{\fk B}^*(\Wbb)$, if and only if $\uppsi(\sigma_v\cdot w)=\uppsi(\sigma_{v,n}^+\cdot w)=\uppsi(\sigma_{v,n}^-\cdot w)=0$ for all $n\geq 1$ and homogeneous $v\in \Vbb$, if and only if 
      \begin{subequations}\label{prime5}
      \begin{gather}
          \label{prime1}\uppsi(Y_+'(v)_{\wt(v)-1}w)=\uppsi(Y_-(v)_{\wt(v)-1}w),\\
          \label{prime2}\uppsi(Y_+'(v)_{\wt(v)+n-1}w)=\uppsi(Y_-(v)_{\wt(v)-n-1}w)=0, \text{ for all }n\geq 1.
      \end{gather}
      \end{subequations}
      
      By \eqref{prime}, for homogeneous $v\in \Vbb$ of weight $a$ and $k\geq 0$,
      \begin{subequations}
\begin{gather}
\label{prime3}Y_\pm(v)_{a-k-1}
=(-1)^a\sum_{j\geq 0}\frac{1}{j!}
Y'_\pm(L(1)^jv)_{a-j+k-1},\\
\label{prime4}Y_\pm'(v)_{a+k-1}
=(-1)^a\sum_{j\geq 0}\frac{1}{j!}
Y_\pm(L(1)^jv)_{a-j-k-1}
\end{gather}
\end{subequations}
Applying \eqref{prime1} to each $L(1)^jv$ gives
\begin{align*}
\uppsi(Y_+(v)_{a-1}w)
&\xlongequal{\eqref{prime3}}(-1)^a\sum_{j\geq0}\frac{1}{j!}
 \uppsi(Y_+'(L(1)^jv)_{a-j-1}w)\\
&\xlongequal{\eqref{prime1}}(-1)^a\sum_{j\geq0}\frac{1}{j!}
 \uppsi(Y_-(L(1)^jv)_{a-j-1}w)\\
&\xlongequal{\eqref{prime4}}\uppsi(Y'_-(v)_{a-1}w),
\end{align*}
which is \eqref{eq18}. Similarly, for $k\geq 1$, 
\begin{align*}
    \uppsi(Y_+(v)_{a-k-1}w)&\xlongequal{\eqref{prime3}} (-1)^a \sum_{j\geq 0}\frac{1}{j!}
 \uppsi(Y_+'(L(1)^jv)_{a-j+k-1}w)\\
 &\xlongequal{\eqref{prime2}}0
\end{align*}
Thus, $\uppsi(Y_+(v)_{a-k-1}w)=\uppsi(Y_-(v)_{a-k-1}w)=0$ for $k\geq 1$. This is \eqref{eq16}. Similarly, one can show that \eqref{eq17} implies \eqref{prime5}. This completes our proof.
\end{proof}
\begin{rem}
    We can also use nodal chiral Lie algebras to define nodal conformal blocks. See \cite[\S~4]{DGT2}. Then the local description of chiral Lie algebra near the node \cite[Prop. 3.3.1]{DGT2} yields the relations \eqref{eq18} and \eqref{eq16} of Prop. \ref{prop19} directly. Although \cite{DGT2}
uses tensor-product labels, the same argument applies to any
$\Wbb\in\Mod(\Vbb^{\otimes2})$, since its two mode actions commute.
\end{rem}

\begin{pp}\label{nodal9}
Recall that $\Vbb$ is $C_2$-cofinite. For each $\Wbb\in \Mod(\Vbb^{\otimes 2})$, $\ST_{\fk B}^*(\Wbb)$, a priori a subspace of $\Wbb^*$, is actually inside $\Wbb'$.
\end{pp}
It follows that $\ST_{\fk B}^*(\Wbb')$ is a subspace of $\Wbb$ for each $\Wbb\in \Mod(\Vbb^{\otimes 2})$.
\begin{proof}
	Since $\Vbb$ is $C_2$-cofinite, it follows from \cite[Thm. 1.16]{GZ2} that the space of nodal conformal blocks \(\ST_{\fk B}^*(\Wbb)\) is finite dimensional. In the special case where
\(\Wbb=\Wbb_1\otimes \Wbb_2\) for some
\(\Wbb_1,\Wbb_2\in \Mod(\Vbb)\), the finite dimensionality of
\(\ST_{\fk B}^*(\Wbb)\) was proved in the analytic setting in
\cite[Rem. 7.5]{Gui-sewingconvergence}, and in the algebraic setting in
\cite[Cor. 4.2]{DGK-presentations}.
	
	Define the action of $L(0)$ on $\ST_{\fk B}^*(\Wbb)$ by 
	\begin{align*}
		\Lan L(0)\uppsi,w\Ran=\Lan \uppsi,L_+(0)w\Ran=\Lan \uppsi,L_-(0)w\Ran,\quad \text{for all }\uppsi\in \ST_{\fk B}^*(\Wbb),w\in \Wbb
	\end{align*}
	where the last equality is due to \eqref{eq18}. We claim that $\ST_{\fk B}^*(\Wbb)$ is invariant under $L(0)$. To see this, for each $\uppsi\in \ST_{\fk B}^*(\Wbb)$, $w\in \Wbb$ and homogeneous $v\in \Vbb$, noting that $L_+(0)$ commutes with $Y_+(v)_{\wt(v)-1}$, we have
	\begin{align*}
		\Lan L(0)\uppsi,Y_+(v)_{\wt(v)-1} w\Ran=\Lan \uppsi,L_+(0)Y_+(v)_{\wt(v)-1} w\Ran=\Lan \uppsi,Y_+(v)_{\wt(v)-1}L_+(0) w\Ran\\
		=\Lan \uppsi,Y_-'(v)_{\wt(v)-1} L_+(0)w\Ran=\Lan \uppsi,L_+(0)Y_-'(v)_{\wt(v)-1} w\Ran=\Lan L(0)\uppsi,Y_-'(v)_{\wt(v)-1} w\Ran
	\end{align*}
	This proves that $L(0)\uppsi$ satisfies \eqref{eq18}. Clearly, $L(0)\uppsi$ satisfies \eqref{eq16}. Thus $\ST_{\fk B}^*(\Wbb)$ is $L(0)$-invariant.
	
	The generalized eigenspace decomposition of $\ST_{\fk B}^*(\Wbb)$ with respect to $L(0)$ is 
	\begin{align*}
		\ST_{\fk B}^*(\Wbb)=\bigoplus_{i=1}^N \ST_{\fk B}^*(\Wbb)_{[\lambda_i]},\quad \text{where }\lambda_i\in \Cbb.
	\end{align*}
	For each $1\leq i\leq N$, $\ST_{\fk B}^*(\Wbb)_{[\lambda_i]}$ is contained in $\Wbb_{[\lambda_i,\lambda_i]}^*$. Hence 
	\begin{align*}
		\ST_{\fk B}^*(\Wbb)\subset \bigoplus_{i=1}^N \Wbb_{[\lambda_i,\lambda_i]}^*\subset \Wbb'.
	\end{align*}
	This completes the proof.
\end{proof}
\begin{rem}
   It follows from the proof that the $C_2$-cofiniteness assumption can be replaced by the finite-dimensionality of the spaces of nodal conformal blocks.
\end{rem}

By Thm. \ref{finitelinear}, $\Mod(\Vbb^{\otimes 2})$ is a finite $\Cbb$-linear category. Moreover, the functor 
\begin{gather}\label{nodalleftexact}
    \Mod(\Vbb^{\otimes 2})\rightarrow \Vect,\qquad
    \Wbb\mapsto \ST_{\fk B}^*(\Wbb')
\end{gather}
is left exact. To see this, let
\[
0\longrightarrow \mathbb{W}_1
\xrightarrow{i}
\mathbb{W}_2
\xrightarrow{p}
\mathbb{W}_3
\longrightarrow 0
\]
be a short exact sequence in $\operatorname{Mod}(\mathbb{V}^{\otimes 2})$.
Precomposition induces
\[
0\longrightarrow
\ST_{\mathfrak{B}}^{*}(\mathbb{W}_3)
\xrightarrow{p^*}
\ST_{\mathfrak{B}}^{*}(\mathbb{W}_2)
\xrightarrow{i^*}
\ST_{\mathfrak{B}}^{*}(\mathbb{W}_1).
\]
Since $p$ is surjective, $p^*$ is injective. If
$\psi\in\ker i^*$, then $\psi$ vanishes on
$\ker p=\operatorname{im} i$, and hence uniquely descends to a linear
functional $\overline{\psi}:\mathbb{W}_3\to\mathbb{C}$ such that
$\psi=\overline{\psi}\circ p$. Since $p$ is a
$\mathbb{V}^{\otimes 2}$-module morphism, the relations
\eqref{eq17} for $\psi$ descend to the corresponding
relations for $\overline{\psi}$. Thus
$\overline{\psi}\in\ST_{\mathfrak{B}}^{*}(\mathbb{W}_3)$, and
therefore
\[
\ker i^*=\operatorname{im} p^*.
\]
Since the contragredient functor is an exact contravariant equivalence, we conclude that the functor \eqref{nodalleftexact} is left exact.

By \cite[Cor. 1.10]{DSPS19-balanced}, every left exact linear functor from a finite $\Cbb$-linear category to $\Vect$ is representable. Therefore, there exists a $\Wbb$-natural linear isomorphism
\begin{align}\label{nodal8}
	\varphi_\Wbb:\Hom_{\Vbb^{\otimes 2}}(\Abb,\Wbb)\xrightarrow{\simeq} \ST_{\fk B}^*(\Wbb')
\end{align}
for some $\Abb=(\Abb,Y_+,Y_-)\in \Mod(\Vbb^{\otimes 2})$. It is clear that the object $\Abb$ realizing such an isomorphism \eqref{nodal8} is unique up to isomorphism. We fix such an isomorphism $\varphi_\Wbb$. We define
\begin{align*}
	\upomega=\varphi_{\Abb}(\id_{\Abb})\in \ST_{\fk B}^*(\Abb')\subset \Abb
\end{align*}
where the last inclusion is due to Prop. \ref{nodal9}. The object $\Abb$ is called the \textbf{(genus 0) nodal fusion product}, and the element $\upomega$ is called the \textbf{canonical conformal block}. The isomorphism $\varphi_\Wbb=\eqref{nodal8}$ is therefore implemented by
\begin{equation}\label{eq14}
	\begin{gathered}
		\varphi_\Wbb:\Hom_{\Vbb^{\otimes 2}}(\Abb,\Wbb)\xrightarrow{\simeq} \ST_{\fk B}^*(\Wbb')\\
        f\mapsto f(\upomega)
	\end{gathered}
\end{equation}
where $f(\upomega)$, a priori an element of $\Wbb$, actually lies in $\ST_{\fk B}^*(\Wbb')$. This is because $f(\upomega)$, viewed as a linear functional $\Wbb'\rightarrow \Cbb$, also satisfies properties similar to those in \eqref{eq17}.
\begin{pp}\label{nodal7}
	The canonical conformal block $\upomega\in \Abb$ generates $\Abb$ as a $\Vbb^{\otimes 2}$-module. 
\end{pp}
\begin{proof}
	Let $\Ubb\in \Mod(\Vbb^{\otimes 2})$ be the submodule of $\Abb$ generated by $\upomega$. We claim that $\Ubb = \Abb$. Suppose, for contradiction, that $\Ubb \neq \Abb$. Take $\Wbb=\Abb/\Ubb$ in \eqref{eq14}. Then \eqref{eq14} specializes to 
	\begin{gather*}
		\varphi_{\Abb/\Ubb}:\Hom_{\Vbb^{\otimes 2}}(\Abb,\Abb/\Ubb)\xrightarrow{\simeq} \ST_{\fk B}^*\big((\Abb/\Ubb)'\big)\\
        f\mapsto f(\upomega)
	\end{gather*}
	Let $\pi:\Abb\rightarrow \Abb/\Ubb$ denote the canonical projection. Since $\Abb/\Ubb$ is nonzero, $\pi$ is a nonzero morphism. However, 
	\begin{align*}
		\varphi_{\Abb/\Ubb}(\pi)=\pi(\upomega)=0
	\end{align*}
	contradicting the injectivity of $\varphi_{\Abb/\Ubb}$. Therefore, $\Ubb=\Abb$, and hence $\upomega$ generates $\Abb$.
\end{proof}

\begin{co}\label{lowest3}
$(\Abb,Y_+)$ is lowest generated as a weak $\Vbb$-module. 
\end{co}
\begin{proof}
	Let $\Omega(\Abb)$ denote the lowest weight subspace of $(\Abb,Y_+)$. By viewing
    \begin{align*}
	\upomega\in \ST_{\fk B}^*(\Abb')\subset \Abb
\end{align*}
     as a conformal block for $\fk B$ and $\Abb'$, we obtain from \eqref{eq16} that
\begin{align*}
    \<\upomega,Y_+(v)_n a'\>=0, \quad \text{if } \wt(v)-n-1>0\text{ and }a'\in \Abb'
\end{align*}
Thus, $Y_+'(v)_n \upomega=0$ if $\wt(v)-n-1>0$, equivalently, $Y_+(v)_n \upomega=0$ if $\wt(v)-n-1<0$. This proves $\upomega\in \Omega(\Abb)$. Since $Y_+$ commutes with $Y_-$, the submodule $\Abb_-$ of $(\Abb,Y_-)$ generated by $\upomega$ is contained in $\Omega(\Abb)$. By Prop. \ref{nodal7}, $\Abb_-$ generates $(\Abb,Y_+)$. Thus, $\Omega(\Abb)$ also generates $(\Abb,Y_+)$.
\end{proof}

\section{Smooth and nodal conformal blocks}

\subsection{Non-equivalence of nodal and smooth conformal block functors}
Recall the 2-pointed smooth connected rational curve $\fn=\eqref{smooth1}=(\Pbb^1|\infty,0;1/\zeta,\zeta)$ and the 2-pointed nodal rational curve $\fk B=\eqref{nodal1}=(\mathcal{B}\big| \infty,0;\xi,\varpi)$, together with their conformal block functors $\ST_\fn^*,\ST_{\fk B}^*$ described in Sec. \ref{smooth3} and \ref{nodal2}.
\begin{thm}\label{main1}
	Assume that there exists a module in $\Mod(\Vbb)$ that is not lowest generated. Then there exist $\Xbb,\Ybb\in \Mod(\Vbb)$ such that
	\begin{align}\label{main2}
		\dim \ST_{\fn}^*(\Xbb\otimes \Ybb)\ne \dim \ST_{\fk B}^*(\Xbb\otimes \Ybb).
	\end{align}
    The picture for \eqref{main2} is
    \begin{align*}
		\dim \ST^*\Bigg(\vcenter{\hbox{{
						   \includegraphics[height=2.5cm]{fig1a.pdf}}}}~\Bigg)\ne 
							\dim \ST^*\Bigg(\vcenter{\hbox{{
											   \includegraphics[height=2.5cm]{fig2.pdf}}}}~\Bigg)
		\end{align*}
	
\end{thm}
\begin{proof}
	Suppose, to the contrary, that we have 
	\begin{align}\label{main5}
		\dim \ST_{\fn}^*(\Xbb\otimes \Ybb)= \dim \ST_{\fk B}^*(\Xbb\otimes \Ybb),\text{ for all }\Xbb,\Ybb\in \Mod(\Vbb)
    \end{align}
	By Prop. \ref{lowest2}, there exists an irreducible $\Wbb\in \Mod(\Vbb)$ such that $P_\Wbb$ is not lowest generated.
	By \eqref{nodal8}, we have  
	\begin{align}\label{main3}
		\dim \ST_{\fk B}^*(\Ubb'\otimes P_\Wbb)=\dim \Hom_{\Vbb^{\otimes 2}}(\Abb,\Ubb\otimes P_\Wbb'),  \text{ for all }\Ubb\in \Mod(\Vbb).
	\end{align}
    By Thm. \ref{finitelinear}, $\Mod(\Vbb)$ is a finite $\Cbb$-linear category. Moreover, the functor 
\begin{gather*}
    \Mod(\Vbb)\rightarrow \Vect,\qquad
    \Ubb\mapsto \Hom_{\Vbb^{\otimes 2}}(\Abb,\Ubb\otimes P_{\Wbb}')
\end{gather*}
is left exact. By \cite[Cor. 1.10]{DSPS19-balanced}, every left exact linear functor from a finite $\Cbb$-linear category to $\Vect$ is representable. Thus, there exists $\Dbb\in \Mod(\Vbb)$ such that we have a $\Ubb$-natural linear isomorphism
	\begin{align}\label{main4}
		\psi_\Ubb:\Hom_{\Vbb}(\Dbb,\Ubb)\xrightarrow{\simeq} \Hom_{\Vbb^{\otimes 2}}(\Abb,\Ubb\otimes P_\Wbb').
	\end{align}
	 We fix such an isomorphism $\psi_\Ubb$. It follows by \eqref{main3} and \eqref{main4} that
	\begin{align}\label{main6}
		\dim \ST_{\fk B}^*(\Ubb'\otimes P_\Wbb)=\dim \Hom_{\Vbb}(\Dbb,\Ubb), \text{ for all }\Ubb\in \Mod(\Vbb).
	\end{align}
    On the other hand, by \eqref{eq8}, $\ST_{\fn}^*(\Ubb'\otimes P_\Wbb)$ can be identified with $\Hom_\Vbb(P_\Wbb,\Ubb)$, so
	\begin{align}\label{main7}
		\dim \ST_{\fn}^*(\Ubb'\otimes P_\Wbb)= \dim \Hom_\Vbb(P_\Wbb,\Ubb),\text{ for all }\Ubb\in \Mod(\Vbb).
	\end{align}
    By \eqref{main5}, \eqref{main6} and \eqref{main7}, we have $\dim \Hom_\Vbb(\Dbb,\Ubb)=\dim \Hom_\Vbb(P_\Wbb,\Ubb)$ for all $\Ubb\in \Mod(\Vbb)$. This, together with Prop. \ref{dimension}, implies that $\Dbb\simeq P_\Wbb$.

	We claim that $\Dbb$ is lowest generated. If this claim is true, then it will contradict our assumption that $P_\Wbb$ is not lowest generated. This completes our proof. To see the claim, let
	\begin{align*}
		\wtd\upalpha \in \Hom_{\Vbb^{\otimes 2}}(\Abb,\Dbb\otimes P_\Wbb')
	\end{align*}
    be the element such that $\wtd\upalpha=\psi_\Dbb(\id_\Dbb)$. The isomorphism $\psi_\Ubb=\eqref{main4}$ is therefore implemented by
	\begin{equation}\label{main18}
	\begin{gathered}
		\psi_\Ubb: \Hom_{\Vbb}(\Dbb,\Ubb)\xrightarrow{\simeq} \Hom_{\Vbb^{\otimes 2}}(\Abb,\Ubb\otimes P_\Wbb')\\
         f\mapsto (f\otimes \id_{P_\Wbb'})\circ \wtd\upalpha
	\end{gathered}
    \end{equation}
	
	We view $\wtd\upalpha$ as a linear map
	\begin{align*}
		\upalpha:\Abb\otimes P_\Wbb\rightarrow \Dbb,\quad a\otimes x\mapsto \Lan \wtd \upalpha(a),x\Ran.
	\end{align*}
	Since $\wtd\upalpha$ intertwines the actions of $\Vbb^{\otimes 2}$, $\upalpha$ satisfies the following property: for each $v\in \Vbb$, $a\in \Abb$ and $x\in P_\Wbb$, the relations 
	\begin{subequations}\label{main15}
	\begin{gather}
		\label{main17}\upalpha(Y_+(v,z)a\otimes x)=Y_\Dbb(v,z)\upalpha(a\otimes x)\\
		\upalpha(Y_-(v,z)a\otimes x)=\upalpha(a\otimes Y_{P_\Wbb}'(v,z)x)
	\end{gather}
\end{subequations}
	hold in $\Cbb[[z^{\pm 1}]]$. Suppose that $\upalpha$ is not surjective. By \eqref{main17}, the image $\upalpha(\Abb\otimes P_\Wbb)$ is a proper left $\Vbb$-submodule of $\Dbb$. Set 
	\begin{align} \label{main16}
	\Ubb=\Dbb/\upalpha(\Abb\otimes P_\Wbb)
     \end{align}
	in \eqref{main18}. Then $\Ubb$ is nonzero and \eqref{main18} takes the form
	\begin{equation}\label{eq15}
	\begin{gathered}
		\Hom_\Vbb(\Dbb,\Dbb/\upalpha(\Abb\otimes P_\Wbb))\xrightarrow{\simeq}  \Hom_{\Vbb^{\otimes 2}}\big(\Abb,(\Dbb/\upalpha(\Abb\otimes P_\Wbb))\otimes P_\Wbb'\big)\\
		f\mapsto (f\otimes \id_{P_\Wbb'})\circ \wtd \upalpha
	\end{gathered}
    \end{equation}
	The map $(f\otimes \id_{P_\Wbb'})\circ \wtd \upalpha$, viewed as a linear map $\Abb\otimes P_\Wbb\rightarrow \Dbb/\upalpha(\Abb\otimes P_\Wbb)$, is equal to $f\circ \upalpha$. Therefore, the image of the nonzero canonical projection $\pi:\Dbb\rightarrow \Dbb/\upalpha(\Abb\otimes P_\Wbb)$ under \eqref{eq15} is zero, contradicting the injectivity of \eqref{eq15}. Therefore, $\upalpha$ must be surjective. 

	Let $\Omega(\Abb)$ denote the lowest weight subspace of $(\Abb,Y_+)$, and let $\Omega(\Dbb)$ denote the lowest weight subspace of $\Dbb$. By \eqref{main17}, we have 
	\begin{align*}
	\upalpha(\Omega(\Abb)\otimes P_\Wbb)\subset\Omega(\Dbb) 
    \end{align*}
	Moreover, by Cor. \ref{lowest3}, $\Omega(\Abb)$ generates $\Abb$ as a left $\Vbb$-module. Since $\upalpha$ is surjective and satisfies \eqref{main17}, it follows that $\Dbb$ is generated by $\upalpha(\Omega(\Abb)\otimes P_\Wbb)$. Consequently, $\Dbb$ is generated by $\Omega(\Dbb)$, and hence is lowest generated. This completes the proof.
\end{proof}
\begin{rem}
	For readers who are familiar with the notion of fusion products (for general compact Riemann surfaces) introduced in \cite{GZ1}, the proof of Thm. \ref{main1} can be explained as follows. The module $\Dbb$ is in fact the fusion product of 
	\begin{align}\label{fusion1}
		\vcenter{\hbox{{
						   \includegraphics[height=2.0cm]{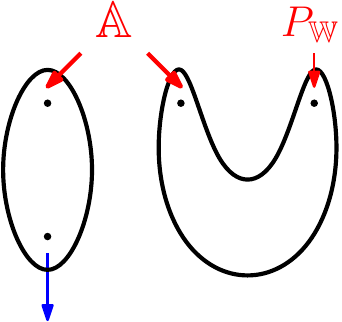}}}}
	\end{align}
    and $\upalpha$ is the canonical conformal block of $\Dbb$:
	\begin{align}\label{fusion2}
		\upalpha\in \ST^*\Bigg(\vcenter{\hbox{{
						   \includegraphics[height=2.5cm]{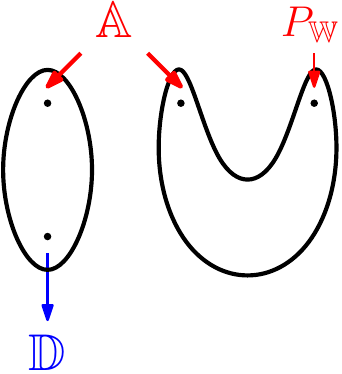}}}}~\Bigg)
	\end{align}
    Here \eqref{main4} is the defining relation for the fusion product in
\eqref{fusion1}, while \eqref{main15} is the conformal block relation for
\eqref{fusion2}. The surjectivity of $\upalpha$ is precisely the partial injectivity property of canonical conformal blocks (cf. \cite[Rem. 3.17]{GZ2}). Nevertheless, since we do not assume that the reader is familiar with \cite{GZ1}, a self-contained proof is provided above.
\end{rem}

\begin{rem}\label{locallyfree}
	Assume that there exists a module in $\Mod(\Vbb)$ that is not lowest generated. Let $N\geq 2$. By Thm. \ref{main1} and propagation of conformal blocks \cite{Zhu-global,Cod19,DGT1,GZ1}, there exist $\Xbb,\Ybb\in \Mod(\Vbb)$ such that 
	\begin{align}\label{main14}
		\dim \ST^*\Bigg(\vcenter{\hbox{{
						   \includegraphics[height=2.5cm]{fig5.pdf}}}}~\Bigg)\ne 
							\dim \ST^*\Bigg(\vcenter{\hbox{{
											   \includegraphics[height=2.5cm]{fig4.pdf}}}}~\Bigg)
		\end{align}
		In \eqref{main14}, both curves carry $N$ marked points, partitioned into two groups. The first group consists of two blue marked points, $\infty$ and $0$, inherited from $\fn,\fk B$ respectively. The second group consists of $N-2$ purple marked points, all distinct from the nodes, each associated with a copy of $\Vbb$. 
		
		Though the curves are not stable, they are affine when the marked points are removed. Therefore, the proof of propagation in \cite[Thm. 6.2]{DGT1} (using Riemann-Roch theorem) still applies to the present situation.

		We conclude that:
		\begin{enumerate}
			\item[(a)] There is no vector bundle on $\overline{\mathcal{M}}_{0,N}$ whose fibers are spaces of conformal blocks associated to $\Xbb,\Ybb,\underbrace{\Vbb,\cdots,\Vbb}_{N-2\text{ copies}}$ for $N\geq 4$. 
			\item[(b)] The sheaf of coinvariants associated to $\Xbb,\Ybb,\underbrace{\Vbb,\cdots,\Vbb}_{N-2\text{ copies}}$ on $\overline{\mathcal{M}}_{0,N}$ is not locally free for $N\geq 4$. (Indeed, the sheaf of coinvariants is defined as quotients and therefore commutes with base change. Hence, its fibers are precisely the spaces of coinvariants, whose dual spaces are spaces of conformal blocks. Therefore, \eqref{main14} implies that the sheaf is not locally free on $\overline{\mathcal{M}}_{0,N}$.)
			\item[(c)] The two conformal block functors (cf. Def. \ref{smoothcb} and \ref{nodalcb}) corresponding to  
		    \begin{align*}
				\vcenter{\hbox{{
								   \includegraphics[height=2.0cm]{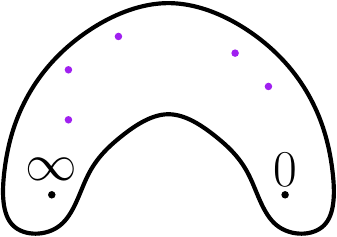}}}}\qquad 
									\vcenter{\hbox{{
													   \includegraphics[height=2.0cm]{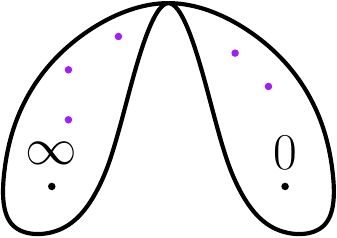}}}}
				\end{align*}
				are not equivalent.
		\end{enumerate}
\end{rem}
\begin{rem}
    Note that taking duals does not, in general, commute with taking fibers. Consequently, the fiber of the sheaf of conformal blocks at a point need not coincide with the space of conformal blocks associated with the corresponding curve. 
    
    Therefore, we \textit{cannot} conclude from Rem. \ref{locallyfree}-(a) that the sheaf of conformal blocks associated to $\Xbb,\Ybb,\underbrace{\Vbb,\cdots,\Vbb}_{N-2\text{ copies}}$ on $\overline{\mathcal{M}}_{0,N}$ is not locally free for $N\geq 4$. 
\end{rem}

\subsection{The end is not isomorphic to the mode transition algebra}
The mode transition algebra $\fk A$ was first introduced in \cite{DGK-presentations}. Its underlying $\Vbb^{\otimes 2}$-module is the object
$\Phi(A(\Vbb))$, where $A(\Vbb)$ is Zhu's algebra regarded as an
$A(\Vbb)$-bimodule and $\Phi$ is the induction functor from
$A(\Vbb)$-bimodules to $U(\Vbb)^{\otimes 2}$-modules. By \cite[Rem. 2.6]{DGK-presentations}, $\fk A$ is a quotient of $\Mbb^{\otimes 2}$, where $\Mbb$ is a grading-restricted generalized $\Vbb$-module. Thus, $\fk A$ is grading-restricted and $\fk A\in \Mod(\Vbb^{\otimes 2})$. In what follows, we only use the induced
$\Vbb^{\otimes 2}$-module structure and the distinguished element
$1\in\mathfrak A$.

Recall the nodal conformal block functor $\ST_{\fk B}^*$ and the smooth conformal block functor $\ST_{\fn}^*$ described in Sec. \ref{smooth3} and \ref{nodal2}.

\begin{thm}[{\cite[Prop. 3.3]{DGK-presentations}}]\label{smooth2}
	Let $\Xbb,\Ybb\in \Mod(\Vbb)$. The linear map
\begin{gather*}
	\Hom_{\Vbb^{\otimes 2}}(\fk A,\Xbb'\otimes \Ybb')\simeq \ST_{\fk B}^*(\Xbb\otimes \Ybb)\\
	T\mapsto T(1)
\end{gather*}
is an isomorphism. 
\end{thm}
\begin{rem}\label{nodal4}
	By \cite[Prop. 3.3]{DGK-presentations} and propagation of conformal blocks (see Rem. \ref{locallyfree} for details), there are natural equivalences
	\begin{align}\label{nodal3}
		\ST^*\Bigg(\vcenter{\hbox{{
					   \includegraphics[height=2.5cm]{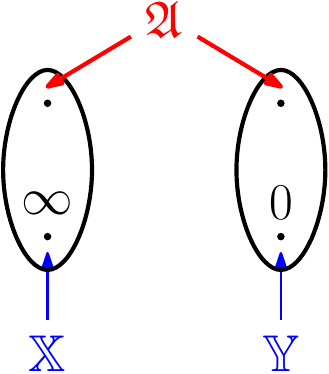}}}}~\Bigg)\xrightarrow{\simeq }\ST^*\Bigg(\vcenter{\hbox{{
						\includegraphics[height=2.5cm]{fig2.pdf}}}}~\Bigg)=\ST_{\fk B}^*(\Xbb\otimes \Ybb)
	\end{align}
The space of smooth conformal blocks on the left hand side of \eqref{nodal3} can be identified with $\Hom_{\Vbb^{\otimes 2}}(\fk A,\Xbb'\otimes \Ybb')$ (see \cite[Prop. 2.3]{GZ5} for details). Therefore, Thm. \ref{smooth2} follows.
\end{rem}

The following remark indicates the relation between the mode transition algebra $\fk A$ and the nodal fusion product $\Abb$.
\begin{rem}\label{mtalowest}
By Thm. \ref{smooth2} and Sec. \ref{nodal2}, the mode transition algebra $\fk A$ and the nodal fusion product $\Abb$ represent the nodal conformal block functor. Thus, there exists an $\Xbb,\Ybb$-natural isomorphism
\begin{align}\label{mtalowest1}
    \Hom_{\Vbb^{\otimes 2}}(\Xbb\otimes \Ybb,\mathfrak A')
\simeq
\Hom_{\Vbb^{\otimes 2}}(\Xbb\otimes \Ybb,\mathbb A').
\end{align}
Let $\Gbb$ be a projective generator of $\Mod(\Vbb)$. By \cite{McR-deligne},
$\Pbb:=\Gbb\otimes_{\mathbb C}\Gbb$ is a projective generator of
$\Mod(\Vbb^{\otimes 2})$. Noting that $\End_{\Vbb^{\otimes 2}}(\Pbb)
\simeq \End_\Vbb(\Gbb)\otimes_{\mathbb C}\End_\Vbb(\Gbb)$, it follows from the $\Xbb,\Ybb$-natural isomorphism \eqref{mtalowest1} that there exists an isomorphism of right $\End_{\Vbb^{\otimes 2}}(\Pbb)$-modules
\begin{align*}
\Hom_{\Vbb^{\otimes 2}}(\Pbb,\mathfrak A')
\simeq
\Hom_{\Vbb^{\otimes 2}}(\Pbb,\mathbb A').
\end{align*}
Since $\Hom_{\Vbb^{\otimes 2}}(\Pbb,-)$ is an equivalence between $\Mod(\Vbb^{\otimes 2})$ and the category of right $\End_{\Vbb^{\otimes 2}}(\Pbb)$-modules, we obtain $\Abb\simeq \fk A$ as objects of $\Mod(\Vbb^{\otimes 2})$. By Cor. \ref{lowest3}, $(\fk A,Y_+)$ is lowest generated.
\end{rem}

On the other hand, we consider the end
\begin{align*}
	\Ebb=\int_{\Mbb\in \Mod(\Vbb)} \Mbb\otimes_\Cbb \Mbb' \quad \in \Mod(\Vbb^{\otimes 2}).
\end{align*}
For each $\Mbb\in \Mod(\Vbb)$, the dinatural transformation of $\Ebb$ gives a morphism $\varphi_\Mbb:\Ebb\rightarrow \Mbb\otimes \Mbb'$ in $\Mod(\Vbb^{\otimes 2})$. 

\begin{pp}\label{main19}
	We have an $\Xbb,\Ybb$-natural isomorphism 
\begin{gather}\label{main20}
	\Hom_{\Vbb^{\otimes 2}}(\Ebb,\Xbb'\otimes \Ybb')\simeq \ST_{\fn}^*(\Xbb\otimes \Ybb)
\end{gather}
\end{pp}
\begin{proof}
    This is already explained in \cite[(0.16) and Thm. 0.1]{GZ3} and \cite[Thm. 2.33]{GZ5}. For the reader's convenience, we provide a complete proof here.

We use the contragredient dual $(-)'$ to identify $\Mod(\Vbb)^\opp$
with $\Mod(\Vbb)$. Under this
identification, the equivalence (cf. \cite{McR-deligne}) $\Mod(\Vbb^{\otimes 2})\simeq
    \Mod(\Vbb)\otimes^{\Del}\Mod(\Vbb)$ will be written as
\begin{align*}
    \Mod(\Vbb^{\otimes 2})\simeq
    \Mod(\Vbb)\otimes^{\Del}\Mod(\Vbb)^\opp .
\end{align*}
Thus we have
\begin{align}\label{deligneend}
    \Ebb=
    \int_{\Xbb\in \Mod(\Vbb)}
    \Xbb\otimes^{\Del}\Xbb^\opp ,
\end{align}
where $\otimes^{\Del}$ denotes the Deligne product.

Since $\Mod(\Vbb)$ is finite $\Cbb$-linear (cf. Thm. \ref{finitelinear}),
there exists a finite dimensional unital $\Cbb$-algebra $A$ such that
$\Mod(\Vbb)\simeq A\text{-mod}$, where $A\text{-mod}$ is the category of
finite dimensional left $A$-modules. Hence
\begin{align*}
    \Mod(\Vbb)\otimes^{\Del}\Mod(\Vbb)^\opp
    \simeq
    A\text{-mod}\otimes^{\Del}(A\text{-mod})^\opp
    \simeq
    A\text{-bimod-}A,
\end{align*}
where $A\text{-bimod-}A$ denotes the category of finite dimensional
$A\text{-}A$-bimodules. Under this equivalence, by \eqref{deligneend}
and \cite[Cor. 2.9]{FSS20}, $\Ebb$ corresponds to the regular
$A\text{-}A$-bimodule $A$.

For $b,c\in A\text{-mod}$, we have
\begin{align*}
    \Hom_{A|A}(A,b\otimes^{\Del}c^\opp)
    \simeq
    \Hom_A(c,b).
\end{align*}
Applying this to the objects corresponding to $\Xbb'$ and $\Ybb$, and
then transporting back to $\Mod(\Vbb)$, we obtain
\begin{align*}
    \Hom_{\Vbb^{\otimes 2}}(\Ebb,\Xbb'\otimes \Ybb')
    \simeq
    \Hom_\Vbb(\Ybb,\Xbb').
\end{align*}
Together with \eqref{eq8}, this gives the isomorphism \eqref{main20}.
\end{proof}

\begin{thm}\label{main8}
	Assume that there exists a module in $\Mod(\Vbb)$ that is not lowest generated. Then $\Ebb\not\simeq \fk A$ in $\Mod(\Vbb^{\otimes 2})$.
\end{thm}
\begin{proof}
By Thm. \ref{main1}, there exist $\Xbb,\Ybb\in \Mod(\Vbb)$ such that 
\begin{align}\label{eq12}
	\dim \ST_{\fn}^*(\Xbb\otimes \Ybb)\ne \dim \ST_{\fk B}^*(\Xbb\otimes \Ybb).
\end{align}
By Thm. \ref{smooth2}, we have an isomorphism 
\begin{align}\label{eq11}
	\Hom_{\Vbb^{\otimes 2}}(\fk A,\Xbb'\otimes \Ybb')\simeq \ST_{\fk B}^*(\Xbb\otimes \Ybb).
\end{align}
Suppose, to the contrary, that $\Ebb\simeq \fk A$ in $\Mod(\Vbb^{\otimes 2})$. By Prop. \ref{main19} and \eqref{eq11}, there exists an isomorphism
\begin{align*}
	\ST_{\fn}^*(\Xbb\otimes \Ybb)\simeq \ST_{\fk B}^*(\Xbb\otimes \Ybb)
\end{align*}
contradicting \eqref{eq12}. Therefore, $\Ebb\not\simeq \fk A$ in $\Mod(\Vbb^{\otimes 2})$.
\end{proof}

The following theorem gives an alternative proof of Thm. \ref{main8}.
\begin{thm}\label{mainlowest}
    Assume that there exists a module in $\Mod(\Vbb)$ that is not lowest generated. Then the weak module $(\Ebb,Y_+)$ is not lowest generated. 
\end{thm}
By Rem. \ref{mtalowest}, $(\mathfrak A,Y_+)$ is lowest generated. Hence Thm. \ref{main8} follows.
\begin{proof}
    By Prop. \ref{lowest2}, there exists an irreducible module $\Wbb$ such that the projective cover $P_\Wbb$ is not lowest generated. By Prop. \ref{main19} and \eqref{eq8}, there exists a $\Ubb$-natural isomorphism
    \begin{align*}
        \psi_\Ubb:\Hom_\Vbb(P_\Wbb,\Ubb)\xrightarrow{\simeq} \Hom_{\Vbb^{\otimes 2}}(\Ebb,\Ubb\otimes P_\Wbb')
    \end{align*}
    Let $\wtd{\upalpha}:\Ebb\rightarrow P_\Wbb\otimes P_\Wbb'$ be the morphism corresponding to $\id_{P_\Wbb}$ under $\psi_{P_\Wbb}$. Define
    \begin{align*}
     \upalpha:\Ebb\otimes P_\Wbb\rightarrow P_\Wbb,\quad a\otimes x \mapsto  \<\wtd{\upalpha}(a),x\>   
    \end{align*}
    For each $v\in \Vbb$, $a\in \Ebb$ and $x\in P_\Wbb$, the relations 
	\begin{subequations}\label{main15a}
	\begin{gather}
		\label{main17a}\upalpha(Y_+(v,z)a\otimes x)=Y_{P_\Wbb}(v,z)\upalpha(a\otimes x)\\
		\upalpha(Y_-(v,z)a\otimes x)=\upalpha(a\otimes Y_{P_\Wbb}'(v,z)x)
	\end{gather}
\end{subequations}
	hold in $\Cbb[[z^{\pm 1}]]$. Repeating the proof of Thm. \ref{main1}, we see that $\upalpha$ is surjective.

 Suppose, for contradiction, that $(\Ebb,Y_+)$ is generated by its lowest weight subspace $\Omega_+(\Ebb)$. Then by \eqref{main17a}, $P_\Wbb$ would be generated by $\upalpha(\Omega_+(\Ebb)\otimes P_\Wbb)$, which is contained in $\Omega(P_\Wbb)$. Hence $P_\Wbb$ is lowest generated, a contradiction.
\end{proof}

\subsection*{Declarations}
Conflict of interest: The author has no conflict of interest to declare that are relevant to the content of this article.

\subsection*{Data availability}
No new data were generated or analyzed in support of this research.

\footnotesize
	\bibliographystyle{alpha}
    \bibliography{voa}

\begin{thebibliography}{EGNO15}

\bibitem[Abe07]{Abe-triplet}
Toshiyuki Abe.
\newblock A {$\mathbb{Z}_2$}-orbifold model of the symplectic fermionic vertex
  operator superalgebra.
\newblock {\em Mathematische Zeitschrift}, 255(4):755--792, 2007.

\bibitem[AM08]{AM-triplet}
Dra{\v{z}}en Adamovi{\'c} and Antun Milas.
\newblock On the triplet vertex algebra {{\(\mathcal W(p)\)}}.
\newblock {\em Adv. Math.}, 217(6):2664--2699, 2008.

\bibitem[BFM91]{BFM-conformal-blocks}
A.~Beilinson, B.~Feigin, and B.~Mazur.
\newblock Introduction to algebraic field theory on curves.
\newblock Unpublished, 1991.

\bibitem[Cod19]{Cod19}
Giulio Codogni.
\newblock Vertex algebras and teichmüller modular forms.
\newblock arXiv:1901.03079, 2019.

\bibitem[DGK24]{DGK3-morita}
Chiara Damiolini, Angela Gibney, and Daniel Krashen.
\newblock Morita equivalences for {Zhu}'s algebra.
\newblock Preprint, {arXiv}:2403.11855 [math.{RT}] (2024), 2024.

\bibitem[DGK25a]{DGK2}
Chiara Damiolini, Angela Gibney, and Daniel Krashen.
\newblock Conformal blocks on smoothings via mode transition algebras.
\newblock {\em Comm. Math. Phys.}, 406(6):Paper No. 131, 58, 2025.

\bibitem[DGK25b]{DGK-presentations}
Chiara Damiolini, Angela Gibney, and Daniel Krashen.
\newblock Factorization presentations.
\newblock In {\em Higher dimensional algebraic geometry---a volume in honor of
  {V}. {V}. {S}hokurov}, volume 489 of {\em London Math. Soc. Lecture Note
  Ser.}, pages 163--191. Cambridge Univ. Press, Cambridge, 2025.

\bibitem[DGT21]{DGT1}
Chiara Damiolini, Angela Gibney, and Nicola Tarasca.
\newblock Conformal blocks from vertex algebras and their connections on
  {$\overline{\mathcal M}_{g, n}$}.
\newblock {\em Geom. Topol.}, 25(5):2235--2286, 2021.

\bibitem[DGT24]{DGT2}
Chiara Damiolini, Angela Gibney, and Nicola Tarasca.
\newblock On factorization and vector bundles of conformal blocks from vertex
  algebras.
\newblock {\em Ann. Sci. \'Ec. Norm. Sup\'er. (4)}, 57(1):241--292, 2024.

\bibitem[DLM98]{DLMtwisted}
Chongying Dong, Haisheng Li, and Geoffrey Mason.
\newblock Twisted representations of vertex operator algebras.
\newblock {\em Mathematische Annalen}, 310(3):571--600, 1998.

\bibitem[DSPS19]{DSPS19-balanced}
Christopher~L. Douglas, Christopher Schommer-Pries, and Noah Snyder.
\newblock The balanced tensor product of module categories.
\newblock {\em Kyoto J. Math.}, 59(1):167--179, 2019.

\bibitem[DW25]{DW-modular-functor}
Chiara Damiolini and Lukas Woike.
\newblock Modular functors from conformal blocks of rational vertex operator
  algebras.
\newblock Preprint, {arXiv}:2507.05845 [math.{QA}] (2025), 2025.

\bibitem[EGNO15]{EGNO}
Pavel Etingof, Shlomo Gelaki, Dmitri Nikshych, and Victor Ostrik.
\newblock {\em Tensor Categories}, volume 205 of {\em Mathematical Surveys and
  Monographs}.
\newblock American Mathematical Society, 2015.

\bibitem[FBZ04]{FB04}
Edward Frenkel and David Ben-Zvi.
\newblock {\em Vertex algebras and algebraic curves}, volume~88 of {\em
  Mathematical Surveys and Monographs}.
\newblock American Mathematical Society, Providence, RI, second edition, 2004.

\bibitem[FGR22]{FGR-symplectic}
V.~Farsad, A.M. Gainutdinov, and I.~Runkel.
\newblock The symplectic fermion ribbon quasi-{Hopf} algebra and the
  {$\mathrm{SL}(2,\mathbb{Z})$}-action on its centre.
\newblock {\em Advances in Mathematics}, 400:108247, 2022.

\bibitem[FSS20]{FSS20}
J\"urgen Fuchs, Gregor Schaumann, and Christoph Schweigert.
\newblock Eilenberg-{W}atts calculus for finite categories and a bimodule
  {R}adford {$S^4$} theorem.
\newblock {\em Trans. Amer. Math. Soc.}, 373(1):1--40, 2020.

\bibitem[GK99]{GaberdielKausch1999}
Matthias~R. Gaberdiel and Horst~G. Kausch.
\newblock A local logarithmic conformal field theory.
\newblock {\em Nuclear Physics B}, 538(3):631--658, 1999.

\bibitem[GR17]{GR-symplectic}
Azat~M. Gainutdinov and Ingo Runkel.
\newblock Symplectic fermions and a quasi-{Hopf} algebra structure on
  {$\bar{U}_i sl(2)$}.
\newblock {\em Journal of Algebra}, 476:415--458, 2017.

\bibitem[Gui24]{Gui-sewingconvergence}
Bin Gui.
\newblock Convergence of sewing conformal blocks.
\newblock {\em Commun. Contemp. Math.}, 26(3):65, 2024.
\newblock Id/No 2350007.

\bibitem[GZ25]{GZ5}
Bin Gui and Hao Zhang.
\newblock How are pseudo-$q$-traces related to (co)ends?
\newblock arXiv:2508.04532, 2025.

\bibitem[GZ26a]{GZ1}
Bin Gui and Hao Zhang.
\newblock Analytic conformal blocks of {$C_2$}-cofinite vertex operator
  algebras {I}: Propagation and dual fusion products.
\newblock {\em New York J. Math.}, 32:859–942, 2026.

\bibitem[GZ26b]{GZ2}
Bin Gui and Hao Zhang.
\newblock Analytic conformal blocks of {$C_2$}-cofinite vertex operator
  algebras {II}: Convergence of sewing and higher genus pseudo-{$q$}-traces.
\newblock {\em Commun. Contemp. Math.}, 28(9), 2026.
\newblock 2550085.

\bibitem[GZ26c]{GZ3}
Bin Gui and Hao Zhang.
\newblock Analytic conformal blocks of {$C_2$}-cofinite vertex operator
  algebras {III}: The sewing-factorization theorems.
\newblock {\em Proc. Lond. Math. Soc.}, 132(2):Paper No. e70130, 2026.

\bibitem[Hua09]{Hua-projectivecover}
Yi-Zhi Huang.
\newblock Cofiniteness conditions, projective covers and the logarithmic tensor
  product theory.
\newblock {\em J. Pure Appl. Algebra}, 213(4):458--475, 2009.

\bibitem[Kau91]{Kausch1991}
Horst~G. Kausch.
\newblock Extended conformal algebras generated by a multiplet of primary
  fields.
\newblock {\em Physics Letters B}, 259(4):448--455, 1991.

\bibitem[Kau95]{Kausch1995}
Horst~G. Kausch.
\newblock Curiosities at $c=-2$, 1995.

\bibitem[Kau00]{Kausch2000}
Horst Kausch.
\newblock Symplectic fermions.
\newblock {\em Nuclear Physics B}, 583(3):513--541, 2000.

\bibitem[McR23]{McR-deligne}
Robert McRae.
\newblock Deligne tensor products of categories of modules for vertex operator
  algebras.
\newblock arXiv:2304.14023, 2023.

\bibitem[MNT10]{MNT10}
Atsushi Matsuo, Kiyokazu Nagatomo, and Akihiro Tsuchiya.
\newblock Quasi-finite algebras graded by {H}amiltonian and vertex operator
  algebras.
\newblock In {\em Moonshine: the first quarter century and beyond}, volume 372
  of {\em London Math. Soc. Lecture Note Ser.}, pages 282--329. Cambridge Univ.
  Press, Cambridge, 2010.

\bibitem[NT05]{NT-P1_conformal_blocks}
Kiyokazu Nagatomo and Akihiro Tsuchiya.
\newblock Conformal field theories associated to regular chiral vertex operator
  algebras. {I}: {Theories} over the projective line.
\newblock {\em Duke Math. J.}, 128(3):393--471, 2005.

\bibitem[NT11]{NagatomoTsuchiya2011}
Kiyokazu Nagatomo and Akihiro Tsuchiya.
\newblock The triplet vertex operator algebra $w(p)$ and the restricted quantum
  group at root of unity.
\newblock In {\em Exploring New Structures and Natural Constructions in
  Mathematical Physics}, volume~61 of {\em Advanced Studies in Pure
  Mathematics}, pages 1--49. American Mathematical Society, 2011.

\bibitem[Run14]{Runkel2014}
Ingo Runkel.
\newblock A braided monoidal category for free super-bosons.
\newblock {\em Journal of Mathematical Physics}, 55(4):041702, 2014.
\newblock 59 pp.

\bibitem[TUY89]{TUY}
Akihiro Tsuchiya, Kenji Ueno, and Yasuhiko Yamada.
\newblock Conformal field theory on universal family of stable curves with
  gauge symmetries.
\newblock Integrable systems in quantum field theory and statistical mechanics,
  {Proc}. {Symp}., {Kyoto}/{Jap}. and {Kyuzeso}/{Jap}. 1988, {Adv}. {Stud}.
  {Pure} {Math}. 19, 459-566 (1989)., 1989.

\bibitem[TW13]{TW-triplet}
Akihiro Tsuchiya and Simon Wood.
\newblock The tensor structure on the representation category of the
  $\mathcal{W}_p$ triplet algebra.
\newblock {\em Journal of Physics A General Physics}, 46, 10 2013.

\bibitem[Uen97]{Ueno97}
Kenji Ueno.
\newblock Introduction to conformal field theory with gauge symmetries.
\newblock In {\em Geometry and physics ({A}arhus, 1995)}, volume 184 of {\em
  Lecture Notes in Pure and Appl. Math.}, pages 603--745. Dekker, New York,
  1997.

\bibitem[Uen08]{Ueno08}
Kenji Ueno.
\newblock {\em Conformal field theory with gauge symmetry}, volume~24 of {\em
  Fields Institute Monographs}.
\newblock American Mathematical Society, Providence, RI; Fields Institute for
  Research in Mathematical Sciences, Toronto, ON, 2008.

\bibitem[Zhu94]{Zhu-global}
Yongchang Zhu.
\newblock Global vertex operators on {Riemann} surfaces.
\newblock {\em Commun. Math. Phys.}, 165(3):485--531, 1994.

\bibitem[Zhu96]{Zhu-modular-invariance}
Yongchang Zhu.
\newblock Modular invariance of characters of vertex operator algebras.
\newblock {\em J. Amer. Math. Soc.}, 9(1):237--302, 1996.

\end{thebibliography}

\noindent {\small \sc Beijing International Center for Mathematical Research, Peking University, Beijing, China.}

\noindent {\textit{E-mail}}: zhanghao1999math@gmail.com 
\end{document}